\documentclass{article}
\usepackage{amssymb}
\usepackage{tikz}
\usepackage{xr}
\usetikzlibrary{arrows, calc, decorations.pathmorphing, backgrounds, positioning, fit, petri, automata}
\definecolor{grey1}{rgb}{0.5,0.5,0.5}
\usepackage[top=2.54cm, bottom=2.54cm, left=2.2cm, right=2.2cm]{geometry}
\usepackage{caption}
\usepackage{subcaption}
\usepackage{booktabs}

\usepackage{algorithmicx}
\usepackage{algcompatible}
\usepackage{algorithm}

\usepackage{hyperref}
\hypersetup{
  colorlinks,
  citecolor=blue,
  linkcolor=blue,
  urlcolor=magenta}

\usepackage{mathtools}
\usepackage{graphicx}
\usepackage{amsmath,amsthm,bm,bbm}
\usepackage{rotating}
\usepackage{mathrsfs}
\usepackage{chngcntr}
\usepackage{array}
\usepackage{apptools}
\usepackage[titletoc,title]{appendix}
\usepackage{comment}

\usepackage{xcolor}         
\definecolor{grau}{rgb}{0.8,0.8,0.8}

\newcommand{\chen}[1]{\color{red}}
\usepackage{setspace,lscape,longtable}
\usepackage{amsmath,amsthm,bm}
\usepackage{fullpage}  

\usepackage{titlesec}
\usepackage[english]{babel}
\usepackage{xcolor}         
\definecolor{grau}{rgb}{0.8,0.8,0.8}

\newtheoremstyle{mystyle}
{1ex} 
{1ex} 
{\itshape} 
{} 
{\bfseries} 
{} 
{1ex} 
{} 
\theoremstyle{plain}
\newtheorem{proposition}{Proposition}
\newtheorem{theorem}{Theorem}
\newtheorem{lemma}{Lemma}
\newtheorem{corollary}{Corollary}

\theoremstyle{definition}

\newtheorem*{example}{Example}

\newtheorem{remark}{Remark}

\usepackage[comma,authoryear]{natbib}

\bibpunct{(}{)}{;}{a}{,}{,}

\DeclareMathOperator*{\argmin}{arg\,min}
\DeclareMathOperator*{\argmax}{arg\,max}
\DeclareMathOperator*{\arginf}{arg\,inf}

\newcommand{\prob}{{\mathbb{P}}}

\newcommand{\expect}{\mathbb{E}}

\newcommand{\transpose}{^{\mathrm{T}}}

\newcommand{\calA}{{\mathcal{A}}}

\newcommand{\calC}{{\mathcal{C}}}
\newcommand{\calD}{{\mathcal{D}}}
\newcommand{\calE}{{\mathcal{E}}}
\newcommand{\calF}{{\mathcal{F}}}

\newcommand{\calN}{{\mathcal{N}}}
\newcommand{\calO}{{\mathcal{O}}}

\newcommand{\calS}{{\mathcal{S}}}

\newcommand{\calU}{{\mathcal{U}}}
\newcommand{\calV}{{\mathcal{V}}}

\newcommand{\calX}{{\mathcal{X}}}

\newcommand{\bU}{{\mathbf{U}}}
\newcommand{\bP}{{\mathbf{P}}}

\newcommand{\bY}{{\mathbf{Y}}}

\newcommand{\bx}{{\mathbf{x}}}
\newcommand{\bX}{{\mathbf{X}}}
\newcommand{\bA}{{\mathbf{A}}}
\newcommand{\bB}{{\mathbf{B}}}
\newcommand{\bC}{{\mathbf{C}}}

\newcommand{\bE}{{\mathbf{E}}}

\newcommand{\bL}{{\mathbf{L}}}

\newcommand{\bR}{{\mathbf{R}}}
\newcommand{\bS}{{\mathbf{S}}}
\newcommand{\bW}{{\mathbf{W}}}
\newcommand{\bZ}{{\mathbf{Z}}}

\newcommand{\bu}{{\mathbf{u}}}

\newcommand{\be}{{\mathbf{e}}}

\newcommand{\bV}{{\mathbf{V}}}

\newcommand{\bDelta}{{\bm{\Delta}}}

\newcommand{\bSigma}{{\bm{\Sigma}}}

\newcommand{\eye}{{\mathbf{I}}}

\newcommand{\zero}{{\bm{0}}}
\newcommand{\eps}{\epsilon}

\usepackage{lipsum}
\usepackage{enumitem}
\setlist[enumerate]{itemsep=0mm}
\setlist[itemize]{itemsep=0mm}

\author{Fangzheng Xie\thanks{Department of Applied Mathematics and Statistics, Johns Hopkins University} \and Yanxun Xu\footnotemark[1] \thanks{Correspondence should be addressed to Yanxun Xu (yanxun.xu@jhu.edu)}
} 
\title{\bf Optimal Bayesian Estimation for Random Dot Product Graphs}
 \date{}
\begin{document}
\allowdisplaybreaks

\maketitle

\begin{abstract}
We propose a Bayesian approach, called the posterior spectral embedding, for estimating the latent positions in random dot product graphs, and prove its optimality. Unlike the classical spectral-based adjacency/Laplacian spectral embedding, the posterior spectral embedding is a fully-likelihood based graph estimation method taking advantage of the Bernoulli likelihood information of the observed adjacency matrix. We develop a minimax-lower bound for estimating the latent positions, and show that the posterior spectral embedding achieves this lower bound since it both results in a minimax-optimal posterior contraction rate, and yields a point estimator achieving the minimax risk asymptotically. 
The convergence results are subsequently applied to clustering in stochastic block models,  the result of which strengthens an existing result concerning the number of mis-clustered vertices. 
We also study a spectral-based Gaussian spectral embedding as a natural Bayesian analogy of the adjacency spectral embedding, but the resulting posterior contraction rate is sub-optimal with an extra logarithmic factor. The practical performance of the proposed methodology is illustrated through extensive synthetic examples and the analysis of a Wikipedia graph data.
\end{abstract}

\noindent%
\section{Introduction} 
\label{sec:introduction}
Using graphs as a  data structure to represent network data with the vertices denoting entities and the edges encoding relationships between vertices, has become increasingly important in a broad range of applications, including social networks \citep{young2007random}, brain imaging \citep{priebe2017semiparametric}, and neuroscience \citep{7769223,tang2016connectome}. 
For example, in a facebook network, vertices represent users, and the occurrence of an edge linking any two users indicates that they are friends on  facebook. 
When one collects  random graph data, it may be costly or even infeasible to collect individual-specific attributes that are heterogeneous across  individuals, while only the adjacency matrix of the graph is accessible. 
For example, in studying the structure of a wikipedia page network, collecting the hyperlinks among articles is much more feasible than collecting the attributes associated with each individual article. To address such a challenge arising in real-world network data, the authors of \cite{doi:10.1198/016214502388618906} proposed \emph{latent positions graphs}, in which each vertex is associated with an unobserved Euclidean vector called the \emph{latent position}, and the edge probability between any two vertices only depends on their latent positions. Formally, each vertex $i$ is associated with a vector $\bx_i$ in some latent space $\calX$, and there exists a symmetric function $\kappa:\calX\times\calX\to [0, 1]$ called \emph{graphon} \citep{lovasz2012large}, such that an edge between vertices $i$ and $j$ occurs with probability $\kappa(\bx_i,\bx_j)$, and the occurrences of these edges are independent given the latent positions. There is vast literature addressing statistical inference on latent positions graphs. For an incomplete list of references, see \cite{Bickel21068,MAL-005,FORTUNATO201075,bickel2011,10.1093/biomet/asr053}, among others.

In this paper we focus on a specific example of latent positions graphs: the \emph{random dot product graph}  model \citep{young2007random}, in which the graphon function $\kappa$ is simply the dot product of two latent positions: $\kappa(\bx_i,\bx_j) = \bx_i\transpose\bx_j$. 
The random dot product graph model enjoys several nice properties. 
Firstly, the well-known stochastic block model, in which the vertices are grouped into several blocks, is a special case of the random dot product graph model and can be represented with the latent positions of vertices in the same block being identical. Secondly, the architecture of the random dot product graph is simple, as the expected value of the adjacency matrix is a symmetric low-rank matrix, motivating the use of a wide range of tractable spectral-based techniques for statistical analysis. Furthermore, the random dot product graph can provide accurate approximation to more general latent positions graphs when the dimension of the latent positions  grows with the number of vertices at a certain rate \citep{tang2013}.   
In addition, its modeling mechanism is convenient for interpretation: In a social network, the latent position for an individual could represent the amount of social activities he/she tends to join, and individuals that are more involved in the same activity are more likely to make acquaintance. 
In light of the structural simplicity and the approximation power of the random dot product graph model, it has become an object that worth studying by itself as well as a useful building block for inferences in more general latent positions graphs. 
 For a thorough review of recent advances in statistical inference on the random dot product graph model, the readers are referred to \cite{JMLR:v18:17-448}. 

The techniques for statistical analysis of the random dot product graph model so far have been focusing on spectral methods based on the observed adjacency matrix or its graph Laplacian matrix. For example, the authors of \cite{6565321} proposed to directly estimate the latent positions using the \emph{adjacency spectral embedding}
and proved its consistency. 
In \cite{athreya2016limit} the authors further established the asymptotic distribution of the adjacency spectral embedding. For the normalized graph Laplacian matrix of the adjacency matrix, the authors of \cite{tang2018} developed the asymptotic distribution of 
the \emph{Laplacian spectral embedding}, and made a thorough comparison between the adjacency spectral embedding and the Laplacian spectral embedding under various contexts. The well-developed theory for spectral methods for the random dot product graph model lays a theoretical foundation for a variety of subsequent inference tasks, including spectral clustering for stochastic block models \citep{sussman2012consistent,lyzinski2014,7769223}, vertex classification and nomination \citep{6565321,lyzinski2017consistent, 7769223}, nonparametric graph hypothesis testing \citep{tang2017}, multiple graph inference \citep{levin2017central,wang2017joint,tang2017robust}, and manifold learning \citep{athreya2018estimation}. 

Despite the marvelous success of spectral methods for the random dot product graph model, it remains open whether these spectral estimators are minimax-optimal for estimating the latent positions with respect to suitable loss functions. Taking one more step back, a more fundamental question is: What is the minimax risk for estimating the latent positions, and how can one achieve it by constructing a useful estimator? In this paper we provide a detailed answer to this question. 
Unlike the aforementioned spectral-based approaches, we take advantage of the Bernoulli likelihood information of the observed graph adjacency matrix and design a fully likelihood-based Bayesian approach, referred to as the \emph{posterior spectral embedding}. Not only do we establish a minimax lower bound for estimating the latent positions, but we also show that this lower bound is achievable through the proposed estimation procedure. Specifically, we show that under mild conditions, for a wide class of prior distributions of the latent positions, the posterior spectral embedding both yields the rate-optimal contraction and produces a minimax-optimal point estimator for the latent positions. To the best of our knowledge, our work represents the first effort in the literature of the random dot product graph model that leverages  a likelihood-based Bayesian approach with theoretical guarantee.
In addition, 
as a sample application, we improve an existing result regarding clustering in stochastic block models by showing that the number of mis-clustered vertices can be reduced from $O(\log n)$ \citep{sussman2012consistent} to $O(1)$ using the proposed posterior spectral embedding.  

The rest of this paper is arranged as follows. In Section \ref{sec:preliminaries} we review the random dot product graph model and establish a minimax lower bound for estimating the latent positions. Section \ref{sec:optimal_bayesian_estimation} elaborates on the proposed posterior spectral embedding for the random dot product graph model and its theoretical properties. The application to clustering in stochastic block models is discussed in Section \ref{sec:sample_application_community_detection_in_stochastic_blockmodels}. 
Section \ref{sec:spectral_based_bayesian_estimation} presents the analysis of a spectral-based Gaussian spectral embedding approach that can be treated as a Bayesian analogy of the adjacency spectral embedding. We illustrate the proposed approach through extensive simulation studies and the analysis of a Wikipedia graph dataset in Section \ref{sec:numerical_examples}. Further discussion is provided in Section \ref{sec:discussion}. 

\noindent\textbf{Notations: }We use $\eye_d$ to denote the $d\times d$ identity matrix. For an integer $p$, $1\leq p\leq \infty$ and a $d$-dimensional Euclidean vector $\bx = [x_1,\ldots,x_d]\transpose$, we use $\|\bx\|_p$ to denote its $\ell_p$-norm, and when $p = \infty$, $\|\bx\|_\infty = \max_{k=1,\ldots,d}|x_k|$. For a vector $\bx = [x_1,\ldots,x_p]\transpose\in\mathbb{R}^p$, the vector inequality $\bx\geq 0$ represents $x_k\geq 0$ for $k = 1,\ldots,p$. For an $n\times d$ matrix $\bX$ we use $(\bX)_{*k}$ to denote the $n$-dimensional vector formed by the $k$th column of $\bX$. For a positive integer $n$, we denote $[n]$ the set of integers $[n] = \{1,2,\ldots,n\}$. For any two positive integers $n,d$ with $n\geq d$, $\mathbb{O}(n,d)$ denotes the set of all orthogonal $d$-frames in $\mathbb{R}^d$, \emph{i.e.}, $\mathbb{O}(n, d) = \{\bU\in\mathbb{R}^{n\times d}:\bU\transpose\bU = \eye_d\}$, and when $n = d$, we use the notation $\mathbb{O}(d) = \mathbb{O}(d, d)$. 
The symbols $\lesssim$ and $\gtrsim$ mean the corresponding inequality up to a constant, \emph{i.e.}, $a \lesssim b$ ($a \gtrsim b$, resp.) if $a\leq Cb$ ($a \geq Cb$) for some constant $C > 0$. 
We write $a\asymp b$ if $a\lesssim b$ and $a\gtrsim b$. For a $d\times d$ positive definite matrix $\bDelta$, we use $\lambda_k(\bDelta)$ to denote its $k$th largest eigenvalue, and for any rectangular matrix $\bX$, we use $\sigma_k(\bX)$ to denote its $k$th largest singular value. We say that a sequence of events $(E_n)_{n = 1}^\infty$ occurs ``almost always'', if $\prob(\bigcup_{n = 1}^\infty\bigcap_{k = n}^\infty E_k) = 1$. 

\section{Preliminaries} 
\label{sec:preliminaries}
We first introduce the background of the random dot product graph model. Let the space of $d$-dimensional latent positions be $\calX = \{\bx\in\mathbb{R}^d:\|\bx\|_2\leq 1, \bx\geq0\}$, where $\|\cdot\|_2$ is the $\ell_2$-norm of an Euclidean vector. Let $\bX = [\bx_1,\ldots,\bx_n]\transpose\in\mathbb{R}^{n\times d}$ be an $n\times d$ matrix, where $\bx_1,\ldots,\bx_n\in\calX$ represent the latent positions of $n$ vertices in a graph. A symmetric random binary matrix $\bY=[y_{ij}]_{n\times n}\in\{0,1\}^{n\times n}$ is said to be the adjacency matrix of a random dot product graph with latent position matrix $\bX$, denoted by $\bY\sim\mathrm{RDPG}(\bX)$, if the random variables $y_{ij}\sim\mathrm{Bernoulli}(\bx_i\transpose\bx_j)$ independently, $1\leq i\leq j\leq n$. Namely, 
\[
p(\bY\mid\bX) = \prod_{i\leq j}(\bx_i\transpose\bx_j)^{y_{ij}}(1 - \bx_i\transpose\bx_j)^{1 - y_{ij}}.
\]
\begin{example}[Stochastic block model]
The most popular example of the random dot product graph model is the stochastic block model with a positive semidefinite block probability matrix. Formally, a symmetric random adjacency matrix $\bY$ is drawn from a stochastic block model with a block probability matrix $\bB = [b_{kl}]_{K\times K}\in[0, 1]^{K\times K}$ and a block assignment function $\tau:[n]\to[K]$, denoted by $\bY\sim\mathrm{SBM}(\bB,\tau)$, if the random variables $y_{ij}\sim\mathrm{Bernoulli}(b_{\tau(i)\tau(j)})$ independently for $1\leq i\leq j\leq n$. Namely, vertices in the same block have the same connecting probability. 
When $\bB$ is positive semidefinite with rank $d$ (and hence implicitly $d\leq K$), there exists a matrix $\bL\in\mathbb{R}^{K\times d}$ such that $\bB = \bL\bL\transpose$. By converting the block assignment function $\tau$ into an $n\times K$ matrix $\bZ = [\mathbbm{1}(\tau(i) = k)]_{i\in[n],k\in[K]}$, we obtain $\expect_\bX(\bY) = (\bZ\bL)\transpose(\bZ\bL)$, and therefore, $\mathrm{SBM}(\bB,\tau)$ coincides with $\mathrm{RDPG}(\bX)$ through the reparametrization $\bX = \bZ\bL$. The positive semidefinite stochastic block model will be revisited in Section \ref{sec:sample_application_community_detection_in_stochastic_blockmodels}. 
\end{example}

\begin{remark}[Intrinsic non-identifiability]
We remark that the latent position matrix $\bX$ cannot be uniquely determined by the distribution $\bY\sim\mathrm{RDPG}(\bX)$, \emph{i.e.}, $\bX$ is not identifiable. In fact, for any orthogonal matrix $\bW\in\mathbb{R}^{d\times d}$, the two distributions $\mathrm{RDPG}(\bX)$ and $\mathrm{RDPG}(\bX\bW)$ are identical, since for any $i,j\in[n]$, $\bx_i\transpose\bx_j = (\bW\bx_i)\transpose(\bW\bx_j)$. In addition, any $d$-dimensional random dot product graph model can be embedded into a $d'$-dimensional random dot product graph model for any $d' > d$, in the sense that there exists a $d'$-dimensional latent position matrix $\bX'\in\mathbb{R}^{n\times d'}$, such that the two distributions $\mathrm{RDPG}(\bX)$ and $\mathrm{RDPG}(\bX')$ are identical. The latter source of non-identifiability, however, can be eliminated by requiring the columns of $\bX$ being linearly independent. 
\end{remark}

\begin{remark}[Choice of orthogonal transformation and loss function]
Since the latent position matrix $\bX$ can only be identified up to an orthogonal transformation, one needs to properly rotate any estimator $\widehat\bX$ to align with the underlying true $\bX$. The alignment matrix can be found by the solution to the orthogonal Procrustes problem $\bW^* = \arginf_{\bW}\|\widehat\bX\bW - \bX\|_{\mathrm{F}}$, where the infimum ranges over the set of all orthogonal matrices in $\mathbb{R}^{d\times d}$ \citep{JMLR:v18:17-448}. 
In particular,  $\bW^*$ has a closed-form expression. 
Consequently, in this work we consider the loss function 
\[
L_{\mathrm{F}}(\widehat\bX,\bX) = \frac{1}{n}\inf_{\bW\in\mathbb{O}(d)}\|\widehat\bX - \bX\bW\|_{\mathrm{F}}^2 = \inf_{\bW\in\mathbb{O}(d)}\frac{1}{n}\sum_{i = 1}^n\|\widehat\bx_{i} - \bW\transpose\bx_{i}\|_2^2,
\] 
where $\widehat\bX = [\widehat\bx_1,\ldots,\widehat\bx_{n}]\transpose$. This loss function can also be interpreted as the average error of the estimated latent positions $\widehat\bx_1,\ldots,\widehat\bx_n$ of all $n$ vertices after the appropriate orthogonal alignment. 
\end{remark}

Since the adjacency matrix $\bY$ can be viewed as the sum of a low-rank signal matrix $\bX\bX\transpose$ and a noise matrix $\bE = (e_{ij})_{n\times n}$, the elements of which are centered Bernoulli random variables ($e_{ij}\sim\mathrm{Bernoulli}(\bx_i\transpose\bx_j) - \bx_i\transpose\bx_j$ independently for $1\leq i\leq j\leq n$), 
 the authors of \cite{6565321} argued for embedding the adjacency matrix $\bY$ into $\mathbb{R}^{n\times d}$ by solving the least-squared problem $\widehat\bX = \argmin_{\bX\in\mathbb{R}^{n\times d}}\|\bY - \bX\bX\transpose\|_{\mathrm{F}}^2$. The resulting estimator $\widehat\bX$ is referred to as the \emph{adjacency spectral embedding} of $\bY$ \citep{sussman2012consistent} and  denoted by $\widehat\bX_{\mathrm{ASE}}$. Theoretical properties of the adjacency spectral embedding have been explored \citep{sussman2012consistent,lyzinski2014,7769223}. Notably, the following convergence result of $\widehat\bX_{\mathrm{ASE}}$ was established in \cite{6565321}.

\begin{theorem}[\citealp{6565321}]
\label{thm:Sussman_convergence_ASE}
Suppose $\bY\sim\mathrm{RDPG}(\bX)$ for some $\bX\in\mathbb{R}^{n\times d}$ and $(1/n)\bX\transpose\bX\to \bDelta$ for some positive definite $\bDelta\in\mathbb{R}^{d\times d}$ with distinct eigenvalues $\lambda_1 > \ldots > \lambda_d > 0$ as $n\to\infty$. Assume that there exists $\delta > 0$ such that $\min_{j\neq k}|\lambda_j(\bDelta) - \lambda_k(\bDelta)| > 2\delta$ and $\lambda_d(\bDelta) > 2\delta$. Then with probability greater than $1 - 2(d^2 + 1)/n^2$,
\begin{align}\label{eqn:ASE_rate_Sussman}
\frac{1}{n}\inf_{\bW\in\mathbb{O}(d)}\|\widehat\bX_{\mathrm{ASE}} - \bX\bW\|_{\mathrm{F}}^2\leq \frac{12d^2\log n}{\delta^3 n}.
\end{align}
\end{theorem}
Theorem \ref{thm:Sussman_convergence_ASE}   implies that after an orthogonal alignment of $\widehat\bX_{\mathrm{ASE}}$ towards $\bX$, the adjacency spectral embedding yields a convergence rate 
$L_{\mathrm{F}}(\widehat\bX_{\mathrm{ASE}}, \bX) = o_{\prob}\left({(M_n\log n)}{/n}\right)$
for arbitrary $M_n\to\infty$ ($(M_n)_{n = 1}^\infty$ should be interpreted as a sequence converging to $\infty$ arbitrarily slowly). 
Nevertheless, as will be seen in Section \ref{sec:optimal_bayesian_estimation}, this rate is sub-optimal and, interestingly, can be improved by a Bayes estimator instead. Furthermore, it is unclear what is the minimax risk for estimating the latent position matrix $\bX$ with respect to the loss $L_{\mathrm{F}}(\cdot,\cdot)$ and how to construct a  estimator to achieve the minimax rate, which we will address in this paper. 
The distinct eigenvalues condition will also be relaxed in Section \ref{sec:optimal_bayesian_estimation}. 
We begin approaching our main goal by first establishing the following minimax lower bound.

\begin{theorem}\label{thm:minimax_LB}
Let $\bY\sim\mathrm{RDPG}(\bX)$ for some $\bX = [\bx_1,\ldots,\bx_n]\transpose$, $\bx_1,\ldots,\bx_n\in\calX$. Assume that $d$ is fixed and does not change with $n$. Let $\widehat\bX$ be an estimator of the latent position matrix $\bX$ satisfying $\|\widehat\bX\|_{\mathrm{F}} \lesssim \sqrt{n}$ with probability one. Then
\begin{align}\label{eqn:minmax_LB}
\inf_{\widehat\bX}\sup_{\bX\in\calX^n}\expect_\bX\left\{\frac{1}{n}\inf_{\bW\in\mathbb{O}(d)}\|\widehat\bX - \bX\bW\|_{\mathrm{F}}^2\right\}\gtrsim\frac{1}{n}.
\end{align}
\end{theorem}

The above minimax lower bound does not necessarily result in a minimax rate of convergence for estimating the latent positions. Nevertheless, if we assume the existence of an estimator $\widehat\bX$ with $\expect_\bX\{(1/n)\inf_\bW\|\widehat\bX - \bX\bW\|_{\mathrm{F}}^2\}\lesssim 1/n$ (which will be rigorously proved in Section \ref{sec:optimal_bayesian_estimation}), then simply applying  Markov's inequality yields
$(1/n)\inf_\bW\|\widehat\bX - \bX\bW\|_{\mathrm{F}}^2 = o_\prob(M_n/n)$ for arbitrary sequence $M_n\to\infty$. This observation suggests that the convergence rate derived in \cite{6565321} for the adjacency spectral embedding might be sub-optimal and motivates us to pursue an estimator achieving the minimax lower bound \eqref{eqn:minmax_LB}. 


\section{A Likelihood-based Posterior Spectral Embedding} 
\label{sec:optimal_bayesian_estimation}
Although it is intuitive and computationally convenient to directly estimate the latent position matrix $\bX$ by the popular spectral-based approaches (\emph{e.g.}, the adjacency spectral embedding), the Bernoulli likelihood information of the adjacency matrix is neglected. On the other hand, likelihood-based methods for the random dot product graph model remain under-explored. In particular, neither the existence nor the uniqueness of the maximum likelihood estimator for $\bX$ is addressed.  In this section, we  develop a Bayesian approach for estimating the latent positions by taking advantage of the Bernoulli likelihood information.


Recall that the space of the latent positions is $\calX = \{\bx\in\mathbb{R}^d:\|\bx\|_2\leq 1,\bx\geq 0\}$. Let $\bX_0 = [\bx_{01},\ldots,\bx_{0n}]\transpose$ be the true latent position matrix, and  $\bX = [\bx_1,\ldots,\bx_n]\transpose$ be the latent position matrix to be assigned a prior distribution $\Pi$. Whenever we consider the distribution $\Pi$, $\bX$ is treated as a random matrix taking values in the space $\calX^n = \{\bX = [\bx_1,\ldots,\bx_n]\transpose:\bx_i\in\calX,i = 1,\ldots,n\}$. The prior distribution $\Pi$ on $\bX$ is constructed by assuming that $\bx_1,\ldots,\bx_n$ independently follow a distribution with a density function $\pi_\bx$ supported on $\calX$, and we denote it by $\bX\sim\Pi$. In this work we only require $\pi_\bx$ to be bounded away from $0$ and $\infty$ over $\calX$ (\emph{e.g.}, the uniform distribution on $\calX$). It follows directly from the Bayes formula that the posterior distribution of $\bX$ is 
\begin{align*}
\Pi(\bX\in\calA\mid\bY) = \frac{N_n(\calA)}{D_n},\quad\text{where }
N_n(\calA) = \int_{\calA}\prod_{i\leq j}\frac{p(y_{ij}\mid\bX)}{p(y_{ij}\mid\bX_0)}\Pi(\mathrm{d}\bX),\quad D_n = N_n(\calX),
\end{align*}
and $p(y_{ij}\mid\bX) = (\bx_i\transpose\bx_j)^{y_{ij}}(1 - \bx_i\transpose\bx_j)^{1 - y_{ij}}$, for any measurable set $\calA\subset\calX^n$. Clearly, the posterior distribution of $\bX$  incorporates the Bernoulli likelihood information through the Bayes formula, and we refer to $\Pi(\bX\in\cdot\mid\bY)$ as the \emph{posterior spectral embedding} of $\bX$.

The following theorem, which is the key result of this work, shows that under mild regularity conditions, the posterior contraction of the latent positions is minimax-optimal. The proof is deferred to the supplementary material. 
\begin{theorem}\label{thm:root_n_contraction_BRDPG}
Let $\bY\sim\mathrm{RDPG}(\bX_0)$ for some $\bX_0 = [\bx_{01},\ldots,\bx_{0n}]\transpose\in\mathbb{R}^{n\times d}$, and the prior $\Pi$ be described as above. Assume that $(1/n)(\bX_0\transpose\bX_0)\to\bDelta$ as $n\to\infty$ for some positive definite $\bDelta\in\mathbb{R}^{d\times d}$. If $d$ is fixed (\emph{i.e.}, $d$ does not change with $n$), and $\delta\leq\min_{i,j}\bx_{0i}\transpose\bx_{0j}\leq \max_{i,j}\bx_{0i}\transpose\bx_{0j}\leq 1-\delta$ for some constant $\delta\in(0, 1/2)$ independent of $n$, then there exist 
some large constants $M_1, M_2 > 0$ depending on $\bDelta$ and the prior $\pi_\bx$, such that
\begin{align*}
\expect_0\left\{\Pi\left(\frac{1}{n}\|\bX\bX\transpose - \bX_0\bX_0\transpose\|_{\mathrm{F}}   >\frac{M_1}{\sqrt{n}}\mathrel{\Big|}\bY\right)\right\}&\leq 8\exp\left(-\frac{1}{2}nd\right),\\
\expect_0\left\{\Pi\left(\frac{1}{n}\inf_{\bW\in\mathbb{O}(d)}\|\bX - \bX_0\bW\|^2_{\mathrm{F}}>\frac{M_2}{{n}}\mathrel{\Big|}\bY\right)\right\} &\leq 8\exp\left(-\frac{1}{2}nd\right)
\end{align*}
for sufficiently large $n$.
\end{theorem}
\begin{remark}
The assumption $(1/n)(\bX_0\transpose\bX_0)\to\bDelta$ as $n\to\infty$ in Theorem \ref{thm:root_n_contraction_BRDPG} can be equivalently written as $(1/n)\sum_{i = 1}^n\bx_{0i}\bx_{0i}\transpose\to\bDelta$ as $n\to\infty$ for some positive definite $\bDelta$. An intuitive interpretation of this condition is that the true latent positions $\bx_{01},\ldots,\bx_{0n}$ can be regarded as ``random'' samples drawn from some non-degenerate distribution with a positive definite second-moment matrix $\bDelta$. By the law of large numbers, the ``sample'' version of the second-moment matrix ``converges'' to the ``population'' version of the second-moment matrix. An illustrative example is the stochastic block model: Suppose the distinct latent positions of $\bx_{01},\ldots,\bx_{0n}$ are $\bx_{01}^*,\ldots,\bx_{0K}^*$, and let $n_k = \sum_{i = 1}^n\mathbbm{1}(\bx_{0i} = \bx_{0k}^*)$ be the number of vertices corresponding to the latent position  $\bx_{0k}^*$. Assume that $K$ is fixed, $n_k/n\to \alpha_k > 0$ as $n\to\infty$, and $\alpha_k$'s, $\bx_{0k}^*$'s are fixed (do not vary with $n$), $k = 1,\ldots,K$. Then 
\[
\frac{1}{n}\bX_0\transpose\bX_0
= \sum_{k = 1}^K\sum_{i = 1}^n\mathbbm{1}(\bx_{0i} = \bx_{0k}^*)\bx_{0i}\bx_{0i}\transpose
= \sum_{k = 1}^K\frac{n_k}{n}(\bx_{0k}^*)(\bx_{0k}^*)\transpose{}
\to \sum_{k = 1}^K\alpha_k(\bx_{0k}^*)(\bx_{0k}^*)\transpose{}
\]
as $n\to\infty$. Therefore, with the above assumption, the stochastic block model satisfies this condition provided that $\sum_{k = 1}^K\alpha_k(\bx_{0k}^*)(\bx_{0k}^*)\transpose{}$ is positive definite. 
\end{remark}
Theorem \ref{thm:root_n_contraction_BRDPG}  claims that under appropriate regularity conditions, the posterior spectral embedding yields a rate-optimal posterior contraction for the latent positions in the Bayesian sense. 
The following theorem shows that one can use the posterior spectral embedding to construct a (frequentist) estimator $\widehat\bX$ that exactly achieves the minimax lower bound \eqref{eqn:minmax_LB}.
\begin{theorem}\label{thm:point_estimator}
Let the conditions in Theorem \ref{thm:root_n_contraction_BRDPG} hold, and let constant $M_1 >  0$ be given by Theorem \ref{thm:root_n_contraction_BRDPG}. Consider the posterior mean of the edge probability matrix
\[
\widetilde\bP = \int_{\bX\in\calX^n}\bX\bX\transpose\Pi(\mathrm{d}\bX\mid\bY).
\]
Suppose $\widetilde\bP$ yields spectral decomposition $\widetilde\bP = \sum_{j = 1}^n\widehat\lambda_j\widehat\bu_j$, where $\widehat\lambda_1,\ldots,\widehat\lambda_n$ are eigenvalues of $\widetilde\bP$ arranged in non-increasing order, and $\widehat\bu_1,\ldots,\widehat\bu_n$ are the associated eigenvectors. Let $\widehat\bU = [\widehat\bu_1,\ldots,\widehat\bu_d]$, $\widehat\bS = \mathrm{diag}(\widehat\lambda_1,\ldots,\widehat\lambda_d)$, $\widehat\bX = \widehat\bU\widehat\bS^{1/2}$, and $\bU_0$ be the left-singular vector matrix of $\bX_0$. Then for sufficiently large $n$,
\begin{align}\label{eqn:PSE_convergence}
\expect_0\left\{\frac{1}{n}\inf_{\bW\in\mathbb{O}(d)}\|\widehat\bX - \bX_0\bW\|_{\mathrm{F}}^2\right\}&\lesssim \frac{1}{n}.
\end{align}
Furthermore, for sufficiently large $n$, 
\begin{align}\label{eqn:Eigenvector_convergence}
\prob_0\left(\inf_{\bW\in\mathbb{O}(d)}\|\widehat\bU - \bU_0\bW\|^2_{\mathrm{F}} > 
\frac{128M_1^2d}{\lambda_d^2(\bDelta){n}}
\right)&\leq 2\exp\left(-\frac{1}{4}M_1d\sqrt{n}\right).
\end{align}
\end{theorem}

We briefly compare the results of Theorem \ref{thm:point_estimator} with those in \cite{6565321}. The convergence rate \eqref{eqn:PSE_convergence} shows that $\widehat\bX$ not only achieves the minimax lower bound \eqref{eqn:minmax_LB}, but also yields a convergence rate $(1/n)\inf_\bW\|\widehat\bX - \bX_0\bW\|_{\mathrm{F}}^2 = o_{\prob_0}(M_n/n)$ for any $M_n\to\infty$, improving the rate \eqref{eqn:ASE_rate_Sussman} obtained in \cite{6565321}. 
The convergence rate of the unscaled eigenvectors $\widehat\bU$ given by \eqref{eqn:Eigenvector_convergence} also improves its counterpart in \cite{6565321}, which is explained as follows. Denote $\bU$ the left-singular vector matrix of $\bX$, and $\widehat\bU_{\mathrm{ASE}}$ that of $\widehat\bX_{\mathrm{ASE}}$. Then in \cite{6565321} the authors show that under the assumptions of Theorem \ref{thm:Sussman_convergence_ASE}, there exists a diagonal matrix $\bW$, the diagonal entries of which are either $1$ or $-1$, such that 
\begin{align}\label{eqn:Eigenvector_rate_Sussman}
\prob_0\left(\|(\widehat\bU_{\mathrm{ASE}})_{*k} - (\bW\bU_0)_{*k}\|_2^2> \frac{3\log n}{\delta^2 n}\right)\leq \frac{2(d^2 + 1)}{n^2}
\end{align}
for $k = 1,\ldots,d$. In contrast, the eigenvector estimate $\widehat\bU$ derived using the posterior spectral embedding improves the convergence rate \eqref{eqn:Eigenvector_rate_Sussman}: Not only do we improve the rate from $(\log n)/ n$ to $1/n$, but we also sharpen the large deviation probability from $O(1/n^2)$ to $O(\mathrm{e}^{-c\sqrt{n}})$ for some constant $c > 0$. The distinct eigenvalues condition for $\bDelta$ required in \cite{6565321} is also relaxed. 

\section{Application: Clustering in Stochastic Block Models} 
\label{sec:sample_application_community_detection_in_stochastic_blockmodels}
This section presents an application of the posterior spectral embedding to clustering in stochastic block models. In particular, we show that the result obtained in this section strengthens an existing result related to the number of mis-clustered vertices. In preparation for doing so, let us first review the $K$-means clustering procedure in general (see, for example, \citealp{lloyd1982least}). Suppose that $n$ data points $\widehat\bx_1,\ldots,\widehat\bx_n$ in $\mathbb{R}^d$ are to be assigned into $K$ clusters, and denote $\widehat\bX = [\widehat\bx_1,\ldots,\widehat\bx_n]\transpose\in\mathbb{R}^{n\times d}$ the corresponding data matrix. The $K$-means clustering centroids of $\widehat\bx_1,\ldots,\widehat\bx_n$, represented by an $n\times d$ matrix $C(\widehat\bX)$ with $K$ distinct rows, are given by
\begin{align*}
\bC(\widehat\bX) = \argmin_{\bC\in\calC_K}\| \bC - \widehat\bX\|_{\mathrm{F}},\quad\text{where }
\calC_K = \{\bC\in\mathbb{R}^{n\times d}:\bC\text{ has }K\text{ distinct rows}\}.
\end{align*}
The corresponding cluster assignment function is defined to be any function $\tau(\cdot;\widehat\bX):[n]\to [K]$ such that $\tau(i;\widehat\bX) = \tau(j;\widehat\bX)$ if and only if $\bC(\widehat\bX)_{i*} = \bC(\widehat\bX)_{j*}$. 
Given two cluster assignment functions $\tau_1$ and $\tau_2$, the Hamming distance between $\tau_1$ and $\tau_2$ is defined by $d_H(\tau_1,\tau_2) = \sum_{i = 1}^n\mathbbm{1}\{\tau_1(i) \neq \tau_2(i)\}$. To avoid the labeling issue, we use $\inf_{\sigma\in\calS_K}d_H(\sigma\circ\tau(\cdot;\bX),\tau(\cdot;\bX_0))$ as the measurement for clustering performance, where $\calS_K$ is the set of all permutations in $[K]$.

A clustering procedure for stochastic block models is called consistent if the resulting fraction of mis-clustered vertices is asymptotically zero. Consistent clustering procedure in stochastic block models have been investigated in earlier work, including likelihood-based methods \citep{10.1093/biomet/asr053}, spectral clustering based on the Laplacian spectral embedding \citep{rohe2011}, $K$-means clustering based on the adjacency spectral embedding \citep{sussman2012consistent}, and modularity maximization \citep{Girvan7821}, among others.  In particular, the authors of \cite{sussman2012consistent} argue that by directly applying the $K$-means procedure to the adjacency spectral embedding $\widehat\bX_{\mathrm{ASE}}$ (\emph{i.e.}, replacing the aforementioned $\widehat\bX$ by $\widehat\bX_{\mathrm{ASE}}$), the number of mis-clustered vertices can be upper bounded by $O(\log n)$. In what follows we show that this result can be strengthened by taking advantage of the $\sqrt{n}$-convergence rate of the posterior spectral embedding. 

Our method for clustering is straightforward: Similar to the $K$-means clustering based on $\widehat\bX_{\mathrm{ASE}}$, we directly apply the $K$-means clustering procedure to the posterior samples collected from the posterior spectral embedding. Specifically, for each realization $\bX$ drawn from the posterior spectral embedding, we obtain a cluster assignment function $\tau(\cdot;\bX)$ by applying the aforementioned $K$-means clustering procedure to $\bX$. This results in a posterior distribution of the cluster assignment function $\Pi(\tau(\cdot;\bX)\in\cdot\mid\bY)$, which is induced from the map $\bX\mapsto \tau(\cdot;\bX)$ and the posterior spectral embedding $\Pi(\bX\in\cdot\mid\bY)$. The below theorem shows that we can recover the clustering structure through the $K$-means procedure even when we assume that 
 the working model is the random dot product graph model, which is more general than the stochastic block model. 

\begin{theorem}\label{thm:PSE_clustering}
Assume the conditions in Theorem \ref{thm:root_n_contraction_BRDPG} hold, and let the constants $M_1, M_2 > 0$ be provided by Theorem \ref{thm:root_n_contraction_BRDPG}. Further assume that $\bX_0 =[\bx_{01},\ldots,\bx_{0n}]\transpose$ has $K$ distinct rows $\bx_{01}^*,\ldots,\bx_{0K}^*$ for some $K\in[n]$, they satisfy $\min_{k\neq k'}\|\bx_{0k}^*-\bx_{0k'}^*\|_2 > \xi$ for some $\xi > 0$, and $n_k := \sum_{i = 1}^n\mathbbm{1}\{\bx_{0i} = \bx_{0k}^*\}  \to\infty$ as $n\to\infty$ for all $k\in [K]$.  
Then for sufficiently large $n$,
\begin{align*}
\expect_0\left\{\Pi\left(\inf_{\sigma\in\calS_K}d_H(\sigma\circ\tau(\cdot;\bX_0), \tau(\cdot;\bX)) \geq \frac{16M_2^2}{\xi^2}\mathrel{\Big|}\bY\right)\right\}\leq 8\exp\left(-\frac{1}{2}nd\right).
\end{align*}
Let $\widehat\bU$ be the left-singular vector matrix of $\widehat\bX$ defined in Theorem \ref{thm:point_estimator}, and $\bU_0$ be that of $\bX_0$. Then it almost always holds that 
\[
\inf_{\sigma\in\calS_K}d_H(\sigma\circ\tau(\cdot;\widehat\bU), \tau(\cdot;\bU_0)) \leq \frac{16}{\xi^2}\left\{\frac{8M_1\sqrt{2d}}{\lambda_d(\bDelta)}\right\}^2.
\]
\end{theorem}
\begin{remark}
In \cite{sussman2012consistent}, the authors directly apply the $K$-means clustering procedure to the $\widehat\bX_{\mathrm{ASE}}$, and show that 
$\inf_{\sigma\in\calS_K}d_H(\sigma\circ\tau(\cdot;\widehat\bX_{\mathrm{ASE}}), \tau(\cdot;\bX_0)) \lesssim \log n$ almost always. 
Namely, the number of vertices that are incorrectly clustered is $O(\log n)$ eventually. The results obtained in Theorem \ref{thm:PSE_clustering} is stronger, since it shows that this number can be further reduced to $O(1)$ in the following two senses: If the $K$-means clustering procedure is applied to the posterior samples drawn from the posterior spectral embedding, then with posterior probability tending to one in $\prob_0$-probability, the posterior number of mis-clustered vertices is upper bounded by a constant; If the $K$-means clustering procedure is directly applied to the unscaled left-singular vector $\widehat\bU$ of the point estimator $\widehat\bX$ obtained in Theorem \ref{thm:point_estimator}, then it almost always holds that this number can be upper bounded by a constant as well. 
\end{remark}

\section{A Spectral-based Gaussian Spectral Embedding} 
\label{sec:spectral_based_bayesian_estimation}


We have seen in Sections \ref{sec:optimal_bayesian_estimation} and \ref{sec:sample_application_community_detection_in_stochastic_blockmodels} the advantages of the  posterior spectral embedding over the adjacency spectral embedding for the random dot product graph model. The major difference is that the posterior spectral embedding is a fully likelihood-based approach taking the Bernoulli likelihood information into account, while the adjacency spectral embedding only leverages
the low-rank structure of the expected value of the adjacency matrix $\bX\bX\transpose = \expect_\bX(\bY)$.
Recall that the adjacency spectral embedding $\widehat\bX_{\mathrm{ASE}}$ is the solution to the minimization problem
$\min_{\bX\in\mathbb{R}^{n\times d}}\|\bY - \bX\bX\transpose\|_{\mathrm{F}}^2$. 
Equivalently, we can also view $\widehat\bX_{\mathrm{ASE}}$ as the maximum likelihood estimator of $\bX$ using a Gaussian likelihood function:
\[
\widehat\bX_{\mathrm{ASE}} = \argmin_{\bX\in\mathbb{R}^{n\times d}}\|\bY - \bX\bX\transpose\|_{\mathrm{F}}^2 = \argmax_{\bX\in\mathbb{R}^{n\times d}}\sum_{i = 1}^n\sum_{j = 1}^n\left\{
-\frac{1}{2}\log(2\pi) - \frac{1}{2}(y_{ij} - \bx_i\transpose\bx_j)^2\right\}.
\] 
The above maximum likelihood interpretation of the adjacency spectral embedding through the Gaussian likelihood function motivates us to study a Bayesian analogy of the adjacency spectral embedding, referred to as the \emph{Gaussian spectral embedding}, introduced as follows. 
 
 Assume that $\Pi_G$ is some prior distribution on the latent position matrix $\bX$. We consider  the following pseudo-posterior distribution by taking the Gaussian density as the working model: 
\begin{align}
\begin{split}
\label{eqn:GSE}
\Pi_G(\bX\in\calA\mid\bY) &= \frac{N_n^G(\calA)}{D_n^G},\quad\text{where }
N_n^G(\calA) = \int_{\calA}\prod_{i,j\in[n]}\frac{\phi(y_{ij} - \bx_i\transpose\bx_j)}{\phi(y_{ij} - \bx_{0i}\transpose\bx_{0j})}\Pi_G(\mathrm{d}\bX),\\
D_n^G &= N_n(\mathbb{R}^{n\times d}),
\end{split}
\end{align} 
for any measurable set $\calA\subset\mathbb{R}^{n\times d}$, where $\phi$ is the density function of $\mathrm{N}(0, 1)$. The formulation of \eqref{eqn:GSE} is completely based on the spectral property of $\bY$ and $\expect_\bX(\bY) = \bX\bX\transpose$, and does not incorporate  the Bernoulli likelihood information. We refer to the (pseudo) posterior distribution \eqref{eqn:GSE} as the Gaussian spectral embedding of $\bY$. Observe that when 
\begin{align}\label{eqn:GSE_Gaussian_prior}
\Pi_G(\mathrm{d}\bX) = \prod_{i = 1}^n\left(\frac{1}{\sqrt{2\pi\sigma^2}}\right)^d\exp\left(-\frac{\bx_i\transpose\bx_i}{2\sigma^2}\right)\mathrm{d}\bx_i 
\end{align}
for some $\sigma^2 > 0$, the maximum \emph{a posteriori} estimator of \eqref{eqn:GSE} is the same as the solution to the minimization problem $\min_{\bX\in\mathbb{R}^{n\times d}}\|\bY - \bX\bX\transpose\|_{\mathrm{F}}^2 + (1/2\sigma^2)\|\bX\|_{\mathrm{F}}^2$. In particular, when $\sigma^2\to\infty$ (corresponding to a non-informative flat prior), the maximum \emph{a posteriori} estimator of \eqref{eqn:GSE} coincides with the adjacency spectral embedding $\widehat\bX_{\mathrm{ASE}}$. Therefore, one can heuristically view the Gaussian spectral embedding defined through \eqref{eqn:GSE} as a direct Bayesian analogy of the adjacency spectral embedding. 

\begin{remark}[Generality of the Gaussian spectral embedding]
Recall that the random dot product graph model can be alternatively regarded as a low-rank matrix model: $\bY = \bX\bX\transpose + \bE$ for some low-rank matrix $\bX\bX\transpose$ and some noise matrix $\bE$. Note that in the formulation of the Gaussian spectral embedding, we do not constrain the latent positions $\bx_1,\ldots,\bx_n$ to lie in the space $\calX = \{\bx\in\mathbb{R}^d:\|\bx\|_2\leq 1, \bx\geq0\}$, and do not assume a parametric form for the distribution of the entries of $\bY$. Namely, the Gaussian spectral embedding \eqref{eqn:GSE} is well-defined not only for the random dot product graph model, but also for a more general class of low-rank matrix models. In the theoretical analysis below, we also assume that the sampling model for $\bY$ is a more general low-rank matrix model $\bY = \bX\bX\transpose + \bE$ for some $\bX\in\mathbb{R}^{n\times d}$, and the entries of $\bE$ are only required to be sub-Gaussian. 
\end{remark}

\begin{theorem}
\label{thm:pseudo_posterior_contraction}
Let $\bY\in\mathbb{R}^{n\times n}$ be a symmetric random matrix with $(y_{ij}:1\leq i \leq j\leq n)$ being independent, and let $\expect_0(\bY) = \bX_0\bX_0\transpose$ for some $\bX_0\in\mathbb{R}^{n\times d}$, where $d/n\to 0$. 
Assume that $(1/n)\bX_0\transpose\bX_0\to \bDelta$ as $n\to\infty$ for some positive definite $\bDelta\in\mathbb{R}^{d\times d}$, and the entries of $\bY - \expect_0(\bY)$ are sub-Gaussian, \emph{i.e.}, there exists some constant $\tau > 0$, such that for all $\bA\in\mathbb{R}^{n\times n}$ with $\|\bA\|_{\mathrm{F}}^2 = 1$, and all $t > 0$,
$\prob_0\left(|\mathrm{Tr}\left(\bA\transpose(\bY - \bX_0\bX_0\transpose)\right)| > t\right)\leq \mathrm{e}^{-\tau t^2}$.
Then there exist some $M > 0$ and a constant $C_\tau$ only depending on $\tau$ and $\bDelta$, such that for sufficiently large $n$, 
\begin{align*}
\expect\left\{\Pi_G\left(\frac{1}{n}\inf_{\bW\in\mathbb{O}(d)}\|\bX - \bX_0\bW\|_{\mathrm{F}}^2 > \frac{Md\log n}{n}\mathrel{\Big|}\bY\right)\right\}
\leq 14\exp(-C_\tau M^2n\log n).
\end{align*}
\end{theorem}
On one hand, when the sampling model is restricted to the random dot product graph model, the posterior contraction rate for the latent positions under the Gaussian spectral embedding is slower than the optimal rate $1/n$ with an extra logarithmic factor, while the posterior spectral embedding yields a rate-optimal contraction.  
On the other hand, the Gaussian spectral embedding can be applied to more general low-rank matrix models, while the posterior spectral embedding is specifically designed for the random dot product graph model. In addition, the posterior spectral embedding requires the latent positions $\bx_1,\ldots,\bx_n$ to lie in the space $\calX$. Such a restriction could potentially lead to a cumbersome Markov chain Monte Carlo sampler for posterior inference. In contrast, the Gaussian spectral embedding has no constraint on the latent positions, making the corresponding posterior computation relatively convenient.

\section{Numerical Examples} 
\label{sec:numerical_examples}

In this section, we evaluate the performance of the proposed posterior spectral embedding in comparison with the spectral-based Gaussian/adjacency spectral embedding through synthetic examples and the analysis of a Wikipedia graph dataset. For each of the numerical setups, the posterior inferences for the posterior spectral embedding and the Gaussian spectral embedding are carried out through a standard Metropolis-Hastings sampler with $15000$ iterations, where the first $5000$ iterations are discarded as burn-in, and $1000$ post-burn-in samples are collected every $10$ iterations. Throughout this section, the prior distribution on the latent positions is set to be the uniform distribution $\mathrm{Unif}(\calX)$ for the posterior spectral embedding, and the Gaussian prior in \eqref{eqn:GSE_Gaussian_prior} with $\sigma = 10$ for the Gaussian spectral embedding. 

\subsection{Stochastic Block Models} 
\label{sub:stochastic_block_models}

We first consider stochastic block models with positive semidefinite block probability matrices.  Three simulation setups are considered, and the number of communities $K$ and the unique values of their latent positions $[\bx_{01}^*,\ldots,\bx_{0K}^*]$ are tabulated in Table \ref{table:SBM_simulation_setup}. In each simulation setup, the numbers of vertices in different clusters are drawn from a multinomial distribution with the probability vector $[1/K,\ldots,1/K]\transpose$. 
\begin{table}[htbp!]
  \centering
  \caption{Simulation setup for stochastic block models}
  \begin{tabular}{c c c c}
    $K$ & $K = 3$  & $K = 5$   & $K = 7$\\
    $n$ & $n = 600$& $n = 1000$& $n = 1400$\\
    $[\bx_{01}^*,\ldots,\bx_{0K}^*]$&
    $
    \begin{bmatrix*}
    0.3 & 0.3 & 0.6\\ 
    0.3 & 0.6 & 0.3
    \end{bmatrix*}
    $
    & $
    \begin{bmatrix*}
    0.3 & 0.3 & 0.7 & 0.7 & 0.5\\
    0.3 & 0.7 & 0.3 & 0.7 & 0.5
    \end{bmatrix*}
    $ & $
    \begin{bmatrix*}
    0.2 & 0.2 & 0.2 & 0.5 & 0.5 & 0.5 & 0.7\\
    0.2 & 0.5 & 0.7 & 0.2 & 0.5 & 0.7 & 0.2
    \end{bmatrix*}
    $
  \end{tabular}%
  \label{table:SBM_simulation_setup}
\end{table}%

For the posterior spectral embedding, we compute the point estimator $\widehat\bX$ given in Theorem \ref{thm:point_estimator}. A point estimator for the Gaussian spectral embedding is also obtained in a similar fashion. Note that although the data generating models are stochastic block models, the posterior inferences are performed under the more general random dot product graph models as the working models. We perform the subsequent clustering based on the $K$-means procedure, as described in Section \ref{sec:sample_application_community_detection_in_stochastic_blockmodels}. 

In \cite{doi:10.1080/01621459.1971.10482356} the author suggested using the Rand index to evaluate the performance of clustering. Specifically, given two partitions $\calC_1 = \{c_{11},\ldots,c_{1r}\}$ and $\calC_2 = \{c_{21},\ldots,c_{2s}\}$ of $[n]$ (\emph{i.e.}, for $i = 1,2$, $c_{ij}$'s are disjoint and their union is $[n]$), denote $a$ the number of pairs of elements in $[n]$ that are both in the same set in $\calC_1$ and in the same set in $\calC_2$, and $b$ the number of pairs in $[n]$ that are neither in the same set in $\calC_1$ nor in the same set in $\calC_2$. Then the Rand index $\mathrm{RI}$ is defined as
$\mathrm{RI} = {2(a + b)}/\{n(n - 1)\}$.
The Rand index is a quantity between $0$ and $1$, with a higher value suggesting better accordance between the two  partitions. In particular, when $\calC_1$ and $\calC_2$ are identical up to relabeling,  $\mathrm{RI} = 1$. 

The comparisons of the Rand indices and the embedding errors $(1/n)\inf_\bW\|\widehat\bX - \bX_0\bW\|_{\mathrm{F}}^2$ for the three embedding approaches are tabulated in Table \ref{table:SBM_simulation_RI} and Table \ref{table:SBM_simulation_Error}, respectively. We see that the point estimates of the posterior spectral embedding are superior than the other two competitors in terms of higher Rand indices and lower embedding errors, whereas the point estimates of the Gaussian spectral embedding perform the worst in all three setups. All the three embedding approaches perform better as the number of vertices $n$ increases. In particular, the Gaussian spectral embedding does not produce satisfactory results when $n = 600$ and $n = 1000$, but performs decently well when $n = 1400$. 
\begin{table}
\caption{Stochastic block models: Rand indices of different clustering methods. PSE, the posterior spectral embedding; ASE, the adjacency spectral embedding; GSE, the Gaussian spectral embedding.}
\centering{%
\begin{tabular}{c c c c}
    Method & PSE (Point estimate) & ASE & GSE (Point estimate)\\
    $K = 3$, $n = 600$ & ${\bf 0.9171}$ & $0.9160$ & $0.7826$ \\
    $K = 5$, $n = 1000$ &   ${\bf 0.9584}$ & $0.9558$ & $0.7187$ \\
    $K = 7$, $n = 1400$ &   ${\bf 0.9964}$ & $0.9508$ & $0.9505$ 
  \end{tabular}%
}
\label{table:SBM_simulation_RI}
\end{table}

\begin{table}
  \caption{Stochastic block models: Errors $(1/n)\inf_\bW\|\widehat\bX - \bX\bW\|_{\mathrm{F}}^2$ of different methods. PSE, the posterior spectral embedding; ASE, the adjacency spectral embedding; GSE, the Gaussian spectral embedding.}
  \centering{%
  \begin{tabular}{c c c c}
    Method & PSE (Point estimate) & ASE & GSE (Point estimate)\\
    $K = 3$, $n = 600$ & ${\bf 1.281\times 10^{-2}}$ & $1.560\times 10^{-2}$ & $2.792\times 10^{-2}$ \\
    $K = 5$, $n = 1000$ &   ${\bf 6.851\times 10^{-3}}$ & $8.548\times 10^{-3}$ & $1.418\times 10^{-2}$ \\
    $K = 7$, $n = 1400$ &   ${\bf 3.460\times 10^{-3}}$ & $3.582\times 10^{-3}$ & $4.200\times 10^{-3}$ 
  \end{tabular}%
  }
  \label{table:SBM_simulation_Error}
\end{table}%
We also visualize 
the three embeddings of the observed adjacency matrix for the three setups in Figures \ref{fig:SBM_simulation_embedding_K3}, \ref{fig:SBM_simulation_embedding_K5}, and \ref{fig:SBM_simulation_embedding_K7}, respectively. The estimation errors of the point estimates under the Gaussian spectral embedding can be clearly recognized from the figures when $n = 600$ and $n = 1000$. We also observe that for  the underlying true latent position $[0.7, 0.7]\transpose$ when  $K = 5$, the adjacency spectral embedding and the point estimator  of the Gaussian spectral embedding produce estimates that may stay outside the latent position space $\calX$, whereas the point estimates of the posterior spectral embedding always lie the space $\calX$. This agrees with the fact that the posterior spectral embedding requires the latent positions to stay inside $\calX$, whereas the Gaussian spectral embedding and the adjacency spectral embedding do not have such constraints.  
\begin{figure}[htbp]
  \centerline{\includegraphics[width=.9\textwidth]{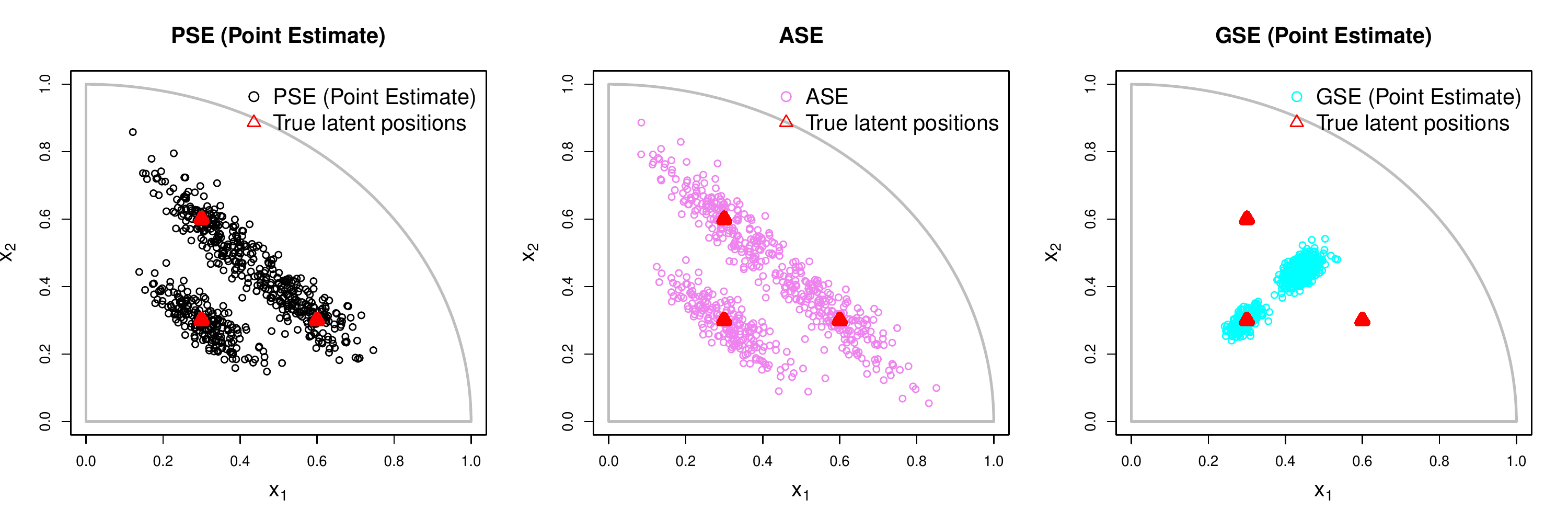}}
  \caption{Visualization of the three embedding approaches for the stochastic block models with $K = 3$; The red triangles are the true latent positions, and the scatter points are embedding estimates of the latent positions. }
  \label{fig:SBM_simulation_embedding_K3}
\end{figure}
\begin{figure}[htbp]
  \centerline{\includegraphics[width=.9\textwidth]{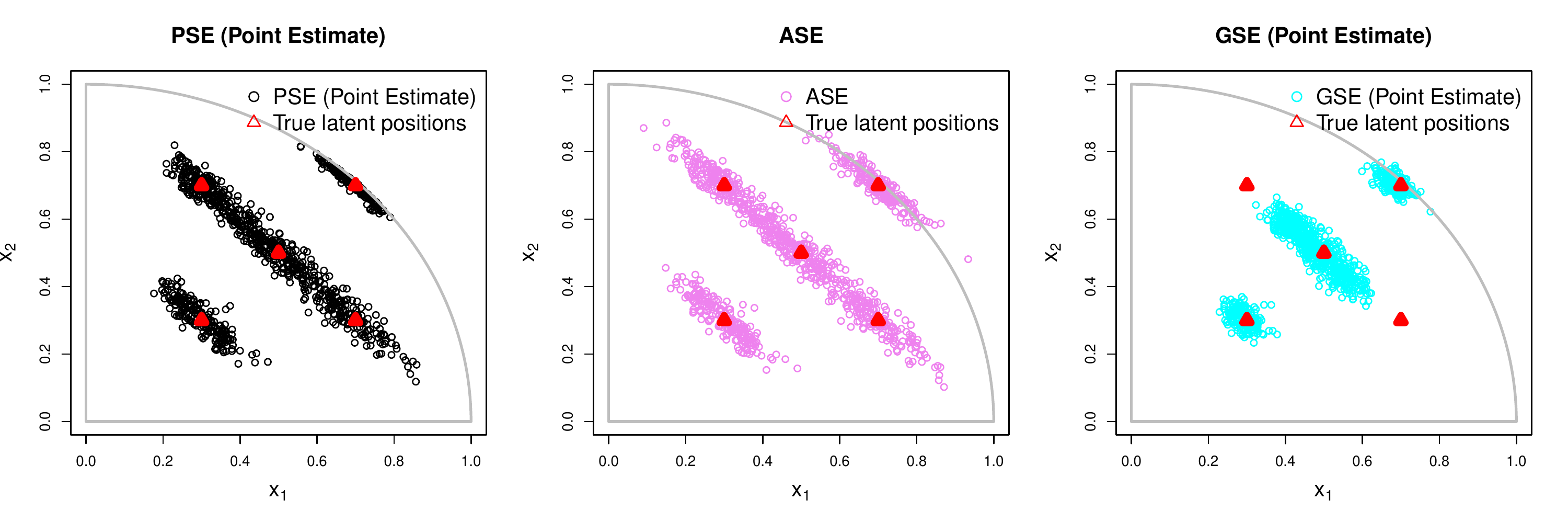}}
  \caption{Visualization of the three embedding approaches for the stochastic block models with $K = 5$; The red triangles are the true latent positions, and the scatter points are embedding estimates of the latent positions. }
  \label{fig:SBM_simulation_embedding_K5}
\end{figure}
\begin{figure}[htbp]
  \centerline{\includegraphics[width=.9\textwidth]{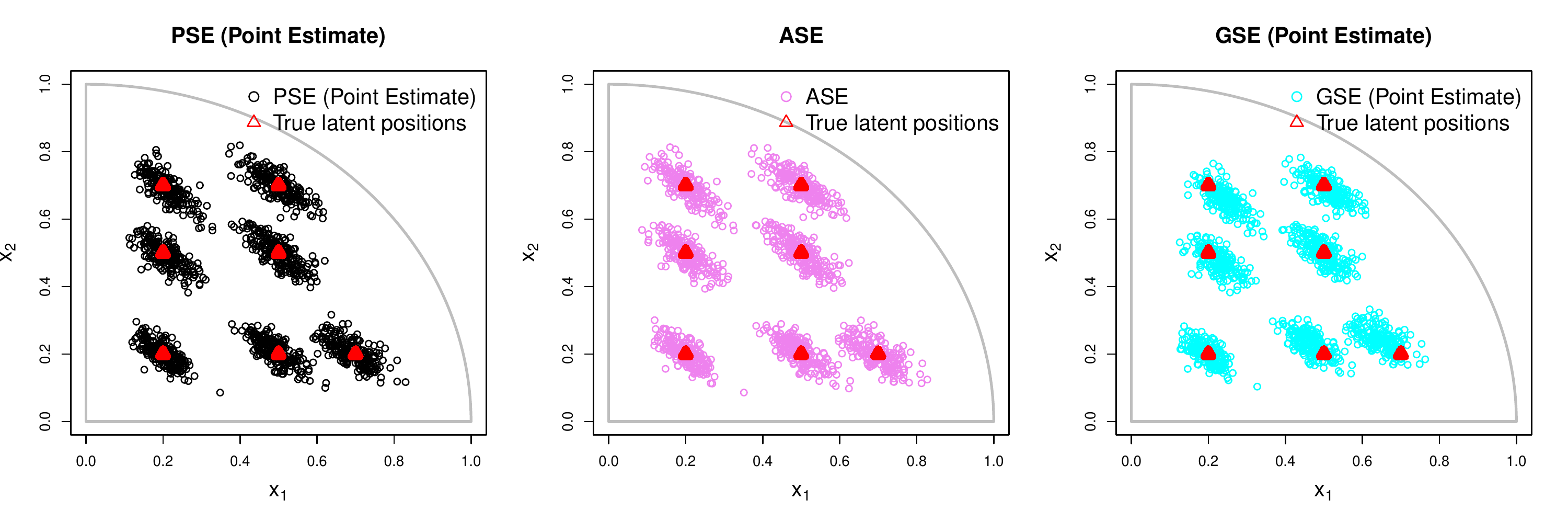}}
  \caption{Visualization of the three embedding approaches for the stochastic block models with $K = 7$; The red triangles are the true latent positions, and the scatter points are embedding estimates of the latent positions. }
  \label{fig:SBM_simulation_embedding_K7}
\end{figure}


\subsection{A Hardy-Weinberg Curve Example} 
\label{sub:latent_structure_models}

We next consider the following Hardy-Weinberg curve example presented in \cite{athreya2018estimation}. Specifically, the observed adjacency matrix $\bY$ is drawn from the random dot product graph model with a latent position matrix  $\bX_0 = [\bx_{01},\ldots,\bx_{0n}]\transpose\in\mathbb{R}^{n\times d}$, where $n = 2000$ and $d = 3$. The latent positions $\bx_{0i}$'s are drawn from the Hardy-Weinberg curve as follows: $\bx_{0i} = [t_i^2, (1 - t_i)^2, 2t_i(1 - t_i)]\transpose\in\mathbb{R}^3$, where $t_1,\ldots,t_n$ are independently drawn from $\mathrm{Unif}(0, 1)$. The latent positions $\bx_{0i}$'s can also be viewed as random samples drawn from the one-dimensional Hardy-Weinberg curve $C(t) = [t^2, (1 - t)^2, 2t(1 - t)]\transpose\in\mathbb{R}^3$, $t\in[0, 1]$, as depicted in Panel (a) of Figure \ref{fig:LSM_Simulation}. We plot 
the embeddings of the observed adjacency matrix under the three approaches 
in panels (b), (c), and (d) of Figure \ref{fig:LSM_Simulation}, showing that 
the point estimates of the posterior spectral embedding produce embeddings of the latent positions that are closer to their true values than the other two competitors do. In particular, the point estimates of the Gaussian spectral embedding are not able to capture the shape of the Hardy-Weinberg curve. The embedding errors $(1/n)\inf_\bW\|\widehat\bX - \bX\bW\|_{\mathrm{F}}^2$ for the three embedding approaches are also presented in Table \ref{table:LSM_simulation_Error}, which is in accordance with the aforementioned observation. 

\begin{figure}
  \centerline{\includegraphics[width=.75\textwidth]{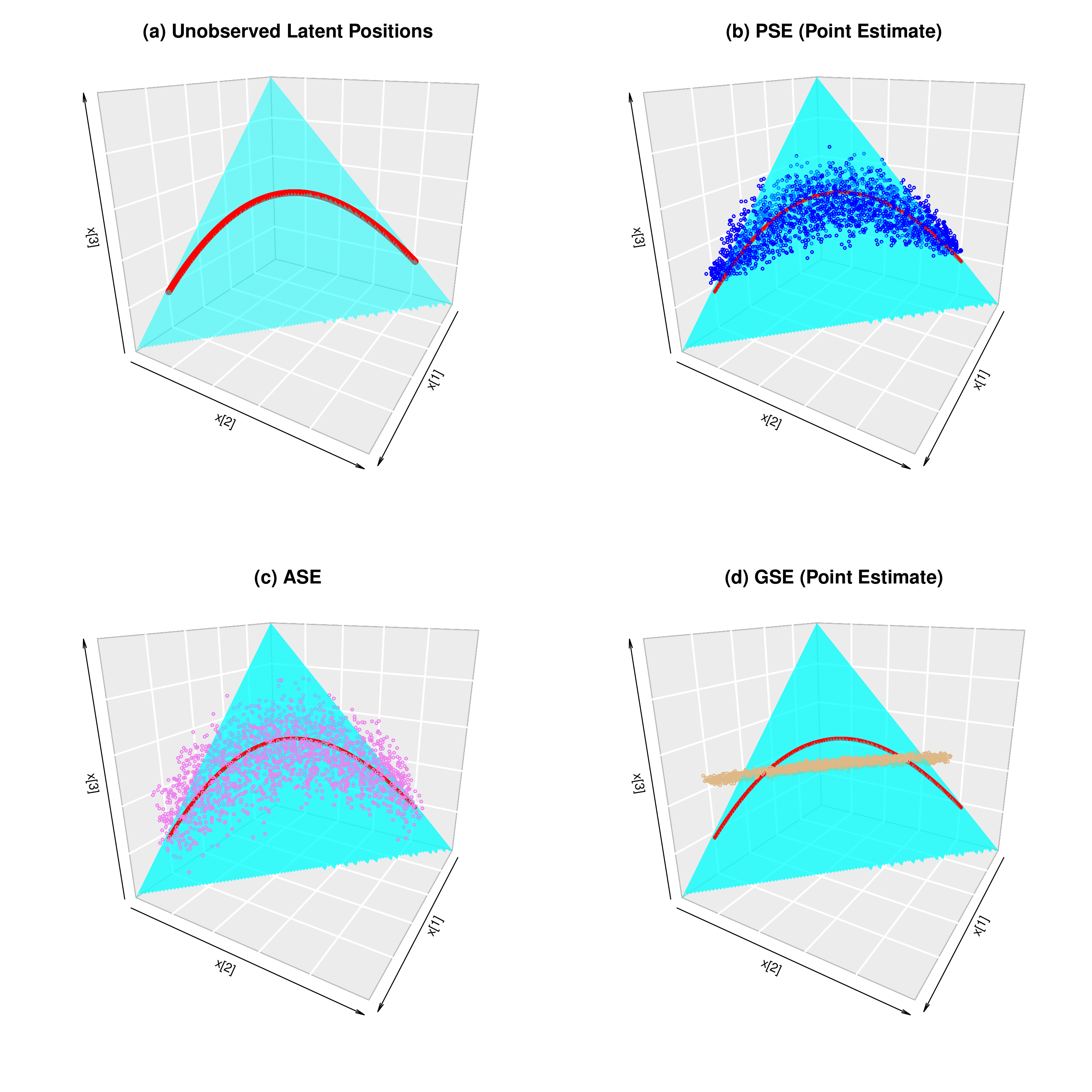}}
  \caption{The Hardy-Weinberg curve example: The scatter points are embedding estimates of the latent positions using the point estimates of the posterior spectral embedding, the adjacency spectral embedding, and the point estimates of the Gaussian spectral embedding, respectively, and the red curve is the underlying unobserved Hardy-Weinburg curve $C(t) = [t^2, (1 - t)^2, 2t(1 - t)]$, $t\in[0, 1]$. }
  \label{fig:LSM_Simulation}
\end{figure}
\begin{table}
  \caption{Hardy-Weinberg curve example: Errors $(1/n)\inf_\bW\|\widehat\bX - \bX\bW\|_{\mathrm{F}}^2$ of different methods. PSE, the posterior spectral embedding; ASE, the adjacency spectral embedding; GSE, the Gaussian spectral embedding.}
  \centering{%
  \begin{tabular}{c c c c}
    Method & PSE (Point estimate) & ASE & GSE (Point estimate)\\
    $(1/n)\inf_\bW\|\widehat\bX - \bX\bW\|_{\mathrm{F}}^2$ &${\bf 9.148\times 10^{-3}}$ & $1.603\times 10^{-2}$ & $1.462\times 10^{-2}$
  \end{tabular}%
  }
  \label{table:LSM_simulation_Error}
\end{table}%

\subsection{Wikipedia Graph Data} 
\label{sub:wikipedia_graph_data}

Our final example is the analysis of a Wikipedia graph dataset available at \url{http://www.cis.jhu.edu/~parky/Data/data.html}. Specifically, the dataset we consider consists of a network of articles that are within two hyperlinks of the article ``Algebraic Geometry'', resulting in $n = 1382$ vertices. In addition, the articles involved are manually labeled as one of the following $6$ classes: People, Places, Dates, Things, Math, and Categories. 

We first estimate the embedding dimension $d$ by an ad-hoc method: We examine the plot of the singular values of the observed adjacency matrix (see Figure \ref{fig:Wikidata_Scree_plot}), and directly locate an ``elbow'' that suggests a cut-off between the signal dimension and the noise dimension. For this Wikipedia dataset, the ``elbow'' is located at $\widehat d = 3$.
\begin{figure}[htbp]
  \centerline{\includegraphics[width=.5\textwidth]{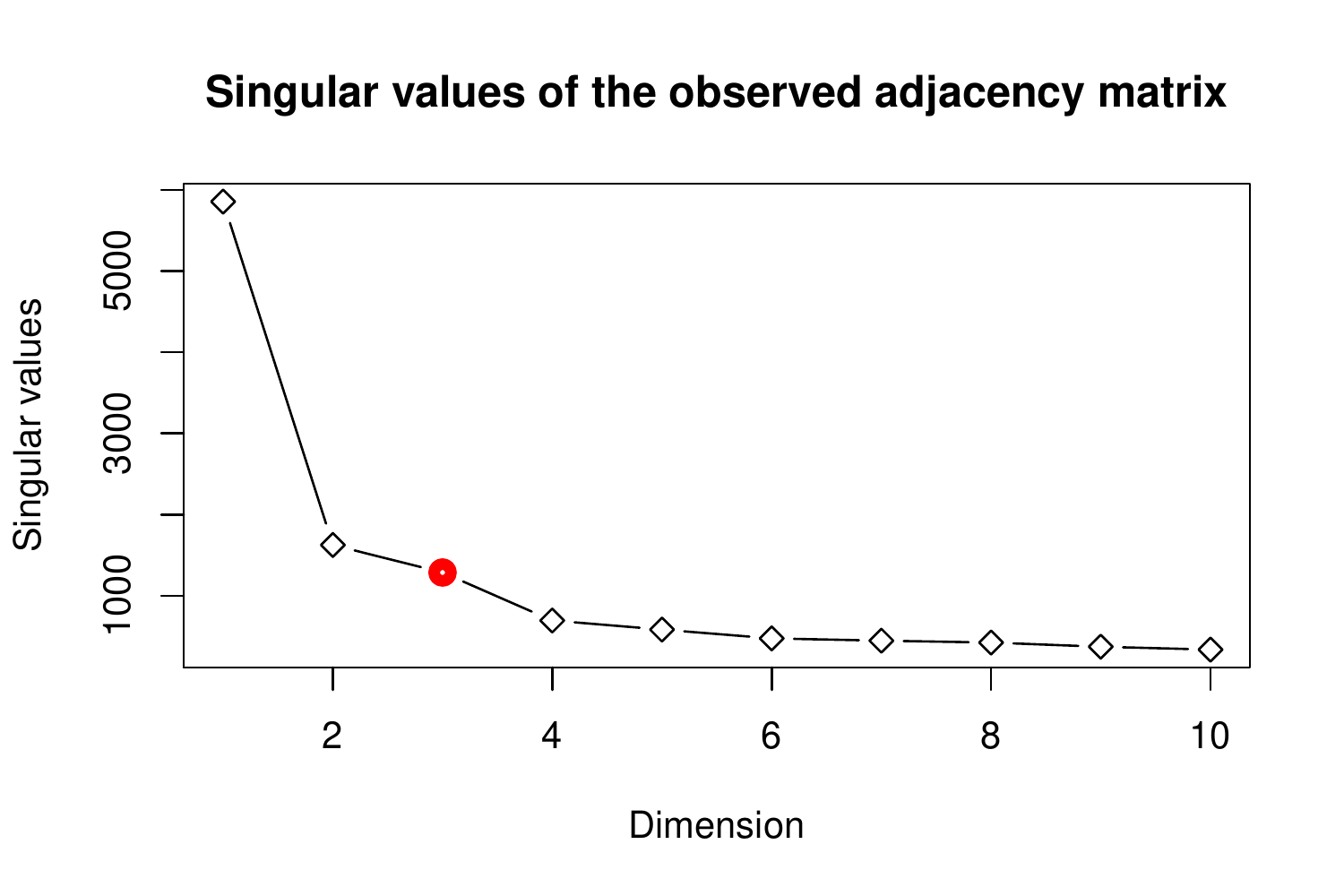}}
  \caption{Wikipedia graph data: The singular values plot of the observed adjacency matrix. An ``elbow'' can be located at $\widehat d = 3$ (the red circle). }
  \label{fig:Wikidata_Scree_plot}
\end{figure}

We then conduct the posterior inferences under the posterior spectral embedding, the Gaussian spectral embedding, along with the adjacency spectral embedding to obtain the estimates of the latent positions based on $\widehat d = 3$. To obtain the clustering results, we further apply the \verb|mclust| package in R \citep{fraley2012mclust} to these embedding estimates with $K = 6$, as discussed in Section \ref{sec:sample_application_community_detection_in_stochastic_blockmodels}, and compute their Rand indices with the manually labeled classes. The results are presented in Table \ref{table:Wikidata_RI}, and we see that the point estimates of the posterior spectral embedding outperform the other two approaches. 
\begin{table}
  \caption{Wikipedia Graph Data: Rand indices of different clustering methods. PSE, the posterior spectral embedding; ASE, the adjacency spectral embedding; GSE, the Gaussian spectral embedding.}
  \centering{%
  \begin{tabular}{c c c c}
    Method & PSE (Point estimator) & ASE & GSE (Point estimator)\\
    Rand Index &  ${\bf 0.7451}$ & $ 0.7213$ & $0.7155$ 
  \end{tabular}%
  }
  \label{table:Wikidata_RI}
\end{table}%


\section{Discussion} 
\label{sec:discussion}
There are several potential extensions of the proposed methodology and the corresponding theory. Firstly, the framework we have considered so far are based on the fact that the observed adjacency matrix of the network are Bernoulli random variables (\emph{i.e.}, a unweighted network). 
It is also common to encounter weighted network data in a wide range of applications \citep{schein2016bayesian,tang2017robust}. Our theory and method tailored for Bernoulli distributed unweighted adjacency matrix can be easily extended to weighted adjacency matrix, the elements of which typically follow distributions of more general forms. In particular, for a weighted adjacency matrix with a specific distribution, the posterior spectral embedding can be generalized similarly to accommodate the corresponding likelihood information. Alternatively, the Gaussian spectral embedding proposed in Section \ref{sec:spectral_based_bayesian_estimation} can be applied when the elements of the weighted adjacency matrix are sub-Gaussian random variables after centering. 
Secondly, the latent positions of the vertices $\bX = [\bx_1,\ldots,\bx_n]\transpose$ are considered as deterministic parameters to be estimated throughout this work. On the other hand, it is also useful to model the latent positions in the random dot product graph model as random variables independently sampled from an underlying distribution $F$ supported on $\calX$ \citep{tang2017}. We can directly apply the technique for estimating $\bX$ in this work to the case where $\bx_i$'s are random, but it requires more effort to explore the theoretical properties 
of the resulting estimator. Last but not least, we assume that the embedding dimension $d$ is known for the ease of the  mathematical analysis in Section \ref{sec:optimal_bayesian_estimation}. When $d$ is unknown, we can first   consistently estimate $d$ by some estimator $\widehat d$ (see, for example, \citealp{chatterjee2015}), and then perform the posterior/Gaussian spectral embedding based on $\widehat d$.  Such a procedure still guarantees the resulting posterior contraction rates. Alternatively, we can assign a prior distribution on $d$ and let the posterior distribution adaptively select the correct dimension with moderate uncertainty. The challenge, nevertheless, is that it is non-trivial to design a reversible-jump sampler to address the cross-dimensional Monte Carlo problem for the random dot product graph model. We defer the computational issue with random $d$ to the future work. On the other hand, the current computational method for the posterior spectral embedding relies on a relatively time-consuming Metropolis-Hastings sampler. The computational cost becomes expensive when the number of vertices grows large. We believe that tackling the computational bottleneck of the posterior spectral embedding will be worthy for a wide range of network data analysis problems as well. 

\clearpage

\begin{center}
	\begin{Large}
		\textbf{Supplementary Material for ``Optimal Bayesian estimation for random dot product graphs''}
	\end{Large}
\end{center}
\appendix
\counterwithin{lemma}{section}
\counterwithin{theorem}{section}

\section{A useful matrix decomposition} 
\label{sec:a_useful_matrix_decomposition}

Before proceeding to the proofs of the main results, we present a slightly technical yet useful matrix decomposition that will be used throughout the supplementary material. 
\begin{lemma}
\label{lemma:Decomposition}
Let $\bX,\bX_0\in\mathbb{R}^{n\times d}$ be $n\times d$ matrices and $\bP_0 = \bX_0\bX_0\transpose$. Let $\bX = \bU\bS^{1/2}\bV\transpose$ and $\bX_0 = \bU_0\bS_0^{1/2}\bV_0\transpose$ be the singular value decomposition of $\bX$ and $\bX_0$, respectively, where $\bU,\bU_0\in\mathbb{O}(n, d)$, $\bV,\bV_0\in\mathbb{O}(d)$, and $\bS^{1/2},\bS_0^{1/2}$ are diagonal matrices with non-negative entries. Further let $\bU_0\transpose{}\bU = \bW_1\bSigma\bW_2\transpose$ be the singular value decomposition of $\bU_0\transpose{}\bU$, where $\bW_1,\bW_2\in\mathbb{O}(d)$, and $\bSigma$ is the diagonal matrix of singular values of $\bU_0\transpose\bU$. Denote $\bW_\bU = \bW_1\bW_2\transpose$. Assume that $[\bU, \bU_\perp]\in\mathbb{O}(n)$, namely, the columns of $\bU_\perp$ are orthonormal and spans the orthogonal complement of $\mathrm{Span}(\bU)$, and $\bP = \bX\bX\transpose + \bU_\perp\bS_\perp\bU_\perp\transpose{}$ for some diagonal $\bS_\perp = \mathrm{diag}(\sigma_{d + 1},\ldots,\sigma_n)$, where $\sigma_1(\bX)\geq\ldots\geq\sigma_d(\bX)\geq \sigma_{d + 1}\geq\ldots\geq \sigma_n$. 
Then the following decomposition holds:
\begin{align*}
\bX\bV - \bX_0\bV_0\bW_\bU & = (\bP - \bP_0)\bU_0\bS_0^{-1/2}\bW_\bU + (\bP - \bP_0)\bU_0(\bW_\bU\bS^{-1/2} - \bS_0^{-1/2}\bW_\bU)\\
&\quad - \bU_0\bU_0\transpose(\bP - \bP_0)\bU_0\bW_\bU\bS^{-1/2} + (\eye - \bU_0\bU_0\transpose)(\bP - \bP_0)\bR_3\bS^{-1/2}\\
&\quad + \bR_1\bS^{1/2} + \bU_0\bR_2,
\end{align*}
where
\begin{align*}
\bR_1 = \bU_0\bU_0\transpose\bU - \bU_0\bW_\bU,\quad
\bR_2 = \bW_\bU\bS^{1/2} - \bS_0^{1/2}\bW_\bU,\quad\text{and}\quad
\bR_3 = \bU - \bU_0\bW_\bU.
\end{align*}
\end{lemma}
\begin{proof}
The proof is similar to that of Theorem 50 in \cite{JMLR:v18:17-448}, and we present it here for completeness. First write
\begin{align*}
\bX\bV - \bX_0\bV_0\bW_\bU & = \bU\bS^{1/2} - \bU_0\bS_0^{1/2}\bW_\bU 
= \bU\bS^{1/2} - \bU_0\bW_\bU\bS^{1/2} + \bU_0\bW_\bU\bS^{1/2} - \bU_0\bS_0^{1/2}\bW_\bU\\
& = (\bU - \bU_0\bU_0\transpose\bU)\bS^{1/2} + (\bU_0\bU_0\transpose\bU - \bU_0\bW_\bU)\bS^{1/2} + \bU_0(\bW_\bU\bS^{1/2} - \bS_0^{1/2}\bW_\bU)\\
& = (\bU\bS^{1/2} - \bU_0\bU_0\transpose\bU\bS^{1/2}) + \bR_1\bS^{1/2} + \bU_0\bR_2\\
& = (\bP\bU\bS^{-1/2} - \bU_0\bU_0\transpose\bP\bU\bS^{-1/2}) + \bR_1\bS^{1/2} + \bU_0\bR_2,
\end{align*}
where the last equality is due to the fact that $\bP\bU\bS^{-1/2} = \bU\bS^{1/2}$. Observe that $\bU_0\bU_0\transpose\bP_0 = \bP_0$, then we re-arrange the term in the parenthesis in the preceding display to
\begin{align*}
&\bP\bU\bS^{-1/2} - \bU_0\bU_0\transpose\bP\bU\bS^{-1/2}\\
&\quad = (\bP - \bP_0)\bU\bS^{-1/2} - \bU_0\bU_0\transpose(\bP - \bP_0)\bU\bS^{-1/2}\\
&\quad = (\bP - \bP_0)\bU\bS^{-1/2} - \bU_0\bU_0\transpose(\bP - \bP_0)(\bU - \bU_0\bW_\bU)\bS^{-1/2}
 - \bU_0\bU_0\transpose(\bP - \bP_0)\bU_0\bW_\bU\bS^{-1/2}\\
&\quad = (\bP - \bP_0)\bU\bS^{-1/2} - \bU_0\bU_0\transpose(\bP - \bP_0)\bR_3\bS^{-1/2}
 - \bU_0\bU_0\transpose(\bP - \bP_0)\bU_0\bW_\bU\bS^{-1/2}\\
&\quad = (\bP - \bP_0)\bU\bS^{-1/2} - (\bP - \bP_0)\bR_3\bS^{-1/2} + (\eye - \bU_0\bU_0\transpose)(\bP - \bP_0)\bR_3\bS^{-1/2}
\\&\quad\quad
 - \bU_0\bU_0\transpose(\bP - \bP_0)\bU_0\bW_\bU\bS^{-1/2}\\
&\quad = (\bP - \bP_0)\bU_0\bW_\bU\bS^{-1/2} + (\eye - \bU_0\bU_0\transpose)(\bP - \bP_0)\bR_3\bS^{-1/2}
 - \bU_0\bU_0\transpose(\bP - \bP_0)\bU_0\bW_\bU\bS^{-1/2}.
\end{align*}
We observe that
\[
(\bP - \bP_0)\bU_0\bW_\bU\bS^{-1/2} = (\bP - \bP_0)\bU_0\bS_0^{-1/2}\bW_\bU + (\bP - \bP_0)\bU_0(\bW_\bU\bS^{-1/2} - \bS_0^{-1/2}\bW_\bU),
\]
and thus complete the proof.
\end{proof}

When the embedding dimension $d$ is one, we obtain immediately the following rank-one corollary:
\begin{corollary}\label{corr:rank_one_corr}
Let $\bx,\bx_0\in(0, 1)^n$ be $n$-dimensional vectors. Denote $\bE = \bx\bx\transpose - \bx_0\bx_0\transpose$. Then the following decomposition holds:
\begin{align*}
\bx - \bx_0 & = \frac{\bE\bx_0}{\|\bx_0\|_2^2} + \left(\frac{1}{\|\bx\|_2} - \frac{1}{\|\bx_0\|_2}\right)\frac{\bE\bx_0}{\|\bx_0\|_2} - \frac{\bx_0\bx_0\transpose\bE\bx_0}{\|\bx\|_2\|\bx_0\|_2^3} + \left(\eye - \frac{\bx_0\bx_0\transpose}{\|\bx_0\|_2^2}\right)\frac{\bE}{\|\bx\|_2}\left(\frac{\bx}{\|\bx\|_2} - \frac{\bx_0}{\|\bx_0\|_2}\right)\\
& \quad + \|\bx\|\left(\frac{\bx_0\bx_0\transpose\bx}{\|\bx_0\|_2^2\|\bx\|_2} - \frac{\bx_0}{\|\bx_0\|_2}\right) + \left(\|\bx\|_2 - \|\bx_0\|_2\right)\frac{\bx_0}{\|\bx_0\|_2}.
\end{align*}
Furthermore, the following inequality holds:
\begin{align*}
\|\bx - \bx_0\|_2 & \leq 
\frac{3\|\bE\|_2}{\|\bx_0\|} + \frac{4\|\bE\|_2}{\|\bx\|_2} + 4\frac{(\|\bx\|_2^2 + \|\bx_0\|_2^2)\|\bE\|_2}{\|\bx_0\|_2^2\|\bx\|_2}.
\end{align*}
\begin{proof}
We first prove the decomposition result. It suffices to show that $\bW_\bU = 1$. In fact, $\bx\transpose\bx_0 > 0$, it follows that $(\bx\transpose\bx_0)/(\|\bx\|_2\|\bx_0\|_2)\in (0, 1)$. Therefore, we can choose the left and right singular vectors of $(\bx\transpose\bx_0)/(\|\bx\|_2\|\bx_0\|_2)$ to be $1$, and consequently, the corresponding orthogonal matrix $\bW_\bU\in\mathbb{R}^{1\times 1}$ is also $1$. 

Now we move forward to prove the inequality result. Observe that
\begin{align*}
&\left\|\left(\frac{1}{\|\bx\|_2} - \frac{1}{\|\bx_0\|_2}\right)\frac{\bE\bx_0}{\|\bx_0\|_2} - \frac{\bx_0\bx_0\transpose\bE\bx_0}{\|\bx\|_2\|\bx_0\|_2^3} + \left(\eye - \frac{\bx_0\bx_0\transpose}{\|\bx_0\|_2^2}\right)\frac{\bE}{\|\bx\|_2}\left(\frac{\bx}{\|\bx\|_2} - \frac{\bx_0}{\|\bx_0\|_2}\right)\right\|_2\\
&\quad \leq \left(\frac{1}{\|\bx\|_2} + \frac{1}{\|\bx_0\|_2}\right)\|\bE\|_2 + \frac{\|\bE\|_2}{\|\bx\|_2} + \frac{2\|\bE\|_2}{\|\bx\|_2} = \frac{\|\bE\|_2}{\|\bx_0\|_2} + \frac{4\|\bE\|_2}{\|\bx\|_2}.
\end{align*}
In addition, by Davis-Kahan theorem and Weyl's inequality,
\begin{align*}
\left\|
\frac{\bx_0\bx_0\transpose\bx}{\|\bx_0\|_2^2\|\bx\|_2} - \frac{\bx_0}{\|\bx_0\|_2}
\right\|_2&\leq \frac{4\|\bE\|_2^2}{\|\bx\|_2^2\|\bx_0\|^2_2}\leq \frac{4(\|\bx\|_2^2 + \|\bx_0\|_2^2)\|\bE\|_2}{\|\bx\|_2^2\|\bx_0\|_2^2},\\
|\|\bx\|_2 - \|\bx_0\|_2|& = \frac{|\|\bx\|_2^2 - \|\bx_0\|_2^2|}{\|\bx\|_2 + \|\bx_0\|_2}\leq\frac{\|\bE\|_2}{\|\bx_0\|}.
\end{align*}
The proof is then completed by combining above derivations. 
\end{proof}
\end{corollary}


\section{Proof of the Minimax Lower Bound} 
\label{sec:minimax_lower_bound}

It is routine to leverage Fano's lemma and its variations to derive minimax lower bounds for a wide class of statistical problems. Specifically, 
we will rely on the following version of Fano's lemma to construct the minimax lower bound for the random dot product graph model. For a totally bounded pseudo-metric space $(T, \rho)$, for any $\eps > 0$, the covering number $\calN(\eps, T, \rho)$ is the minimum number of balls of radius $\eps$ (with respect to the metric $\rho$) that are needed to cover $T$. 

\begin{lemma}[Proposition 3, \citealp{cai2013sparse}]
\label{lemma:Fano_lemma}
Let $(\Theta,\rho)$ be a totally bounded pseudo-metric space and $\{\prob_\theta:\theta\in\Theta\}$ a collection of distributions. Let
$A = \sup_{\theta\neq\theta'}{D(\prob_\theta||\prob_{\theta'})}/{\rho^2(\theta,\theta')}$.
If there exist $0 < c_0 < c_1 < \infty$, $\eps_0 > 0$, and $\alpha \geq1$ such that 
\[
\left(\frac{c_0}{\epsilon}\right)^\alpha\leq \calN(\eps, \Theta, \rho)\leq \left(\frac{c_1}{\eps}\right)^{\alpha}
\]
for all $\eps\in(0, \eps_0)$, then 
\[
\inf_{\hat\theta}\sup_{\theta\in\Theta}\expect_\theta\{\rho^2(\hat\theta,\theta)\} \geq\frac{c_0^2}{840c_1^2}\min\left(\frac{\alpha}{A}, \eps_0^2\right).
\] 
\end{lemma}

We now proceed to the proof of Theorem \ref{thm:minimax_LB}. 
\begin{proof}[of Theorem \ref{thm:minimax_LB}]
Consider the following subset of latent positions:
\begin{align*}
\widetilde{\Theta}_n = \left\{\bX = [\bx_1,\ldots,\bx_n]\transpose\in\calX^n:\sqrt{\frac{1}{4d}}\leq x_{i1}\leq \sqrt{\frac{3}{4d}}, x_{i2} = \ldots = x_{id} = \sqrt{\frac{1}{4d}}, i \in [n]\right\}.
\end{align*}
Let $\bX_1, \bX_2\in\widetilde\Theta_n$, and let $\bu_1,\bu_2$ be the first columns of $\bX_1$ and $\bX_2$, respectively. Clearly, 
\begin{align*}
\|\bX_1\bX_1\transpose - \bX_2\bX_2\transpose\|_{\mathrm{F}}^2 
= \sum_{i = 1}^n\sum_{j = 1}^n(x_{1i1}x_{1j1} - x_{2i1}x_{2j1})^2 = \|\bu_1\bu_1\transpose - \bu_2\bu_2\transpose\|_{\mathrm{F}}^2.
\end{align*}
Since $\bX_1, \bX_2\in\widetilde\Theta_n$, it follows that $\sqrt{n/(4d)}\leq \|\bu_1\|_2,\|\bu_2\|_2\leq \sqrt{3n/(4d)}$. Applying Corollary \ref{corr:rank_one_corr} yields
\[
\|\bu_1 - \bu_2\|_2\leq (14\sqrt{d})\frac{\|\bX_1\bX_1\transpose - \bX_2\bX_2\transpose\|_2}{\sqrt{n}} + \frac{(6n/d)\|\bX_1\bX_1\transpose - \bX_2\bX_2\transpose\|_2}{(n/4d)^{3/2}}\leq\frac{62\sqrt{d}\|\bX_1\bX_1\transpose - \bX_2\bX_2\transpose\|_2}{\sqrt{n}}.
\]
Let $\rho:\widetilde\Theta_n\times\widetilde\Theta_n\to[0,\infty)$ be defined by $\rho(\bX_1,\bX_2) = (1/n)\|\bX_1\bX_1\transpose - \bX_2\bX_2\transpose\|_{\mathrm{F}}$.
It follows that there exists some $\eps_0> 0$, such that for all $\eps\in(0, \eps_0)$,
\begin{align*}
\calN(\eps, \widetilde\Theta_n, \rho)\geq \calN\left(\eps, \left[\sqrt{\frac{1}{4d}},\sqrt{\frac{3}{4d}}\right]^{n}, \frac{1}{62\sqrt{nd}}\|\cdot\|_2\right) = \calN\left(62\sqrt{nd}\eps, \left[\sqrt{\frac{1}{4d}},\sqrt{\frac{3}{4d}}\right]^{n},  \|\cdot\|_2\right).
\end{align*}
For any $\Theta\subset\mathbb{R}^n$, a standard volume comparison argument yields
\[
\left(\frac{1}{\eps}\right)^n\frac{\mathrm{vol}(\Theta)}{\mathrm{vol}(B_1^n)}\leq \calN(\eps,\Theta, \|\cdot\|_2).
\]
Hence,
\[
\calN\left(62\sqrt{nd}\eps, \left[\sqrt{\frac{1}{4d}},\sqrt{\frac{3}{4d}}\right]^{n},  \|\cdot\|_2\right)\geq \left(\frac{1}{62\sqrt{nd}\eps}\right)^n\left(\frac{\sqrt{3} - 1}{2\sqrt{d}}\right)^n\frac{\Gamma(n/2 + 1)}{\pi^{n/2}}.
\]
Observe that by Stirling's formula, for sufficiently large $n$, it holds that
\begin{align}\label{eqn:volume_nball_Stirling}
\left(\frac{n}{2\mathrm{e}}\right)^{n/2}\geq\frac{\Gamma\left(n/2 + 1\right)}{\pi^{n/2}}\geq\left(\frac{n}{2\pi\mathrm{e}}\right)^{n/2}.
\end{align}
Thus, we obtain the following lower bound for the covering number of $\widetilde\Theta_n$:
\[
\calN(\eps, \Theta_n, \rho)\geq\left\{\frac{(\sqrt{3} - 1)}{124\sqrt{2\pi\mathrm{e}}\sqrt{d}\eps}\right\}^n.
\]
We proceed to derive an upper bound for the covering number. Note that
\[
\|\bX_1\bX_1\transpose - \bX_2\bX_2\transpose\|_{\mathrm{F}} = \|\bu_1\bu_1\transpose - \bu_2\bu_2\transpose\|_{\mathrm{F}}\leq \|\bu_1(\bu_1 - \bu_2)\transpose\|_{\mathrm{F}} + \|(\bu_1 - \bu_2)\bu_2\transpose\|_{\mathrm{F}}\leq \sqrt{\frac{3n}{d}}\|\bu_1 - \bu_2\|_2.
\]
This further implies that
\begin{align*}
\calN(\eps, \widetilde\Theta_n, \rho)&
\leq 
\calN\left(\sqrt{\frac{nd}{3}}\eps, \left[\sqrt{\frac{1}{4d}},\sqrt{\frac{3}{4d}}\right]^{n},  \|\cdot\|_2\right)
\leq \left\{\frac{4\sqrt{3}(\sqrt{3} - 1)}{\sqrt{2\mathrm{e}}\sqrt{d}\eps}\right\}^n
\end{align*}
by a simple volume comparison argument. Hence, we obtain the following estimate of the covering number:
\begin{align}
\label{eqn:covering_number_RDPG}
\left(\frac{c_0}{\sqrt{d}\eps}\right)^n\leq \calN(\eps, \widetilde\Theta_n, \rho)\leq \left(\frac{c_1}{\sqrt{d}\eps}\right)^n
\end{align}
for some constants $0 < c_0 < c_1 < \infty$. It remains to derive 
\[A = \sup_{\bX_1,\bX_2\in\widetilde\Theta_n}\{D(\prob_{\bX_1}||\prob_{\bX_2})/\rho^2(\bX_1,\bX_2)\}.\]
For any $\bX_1,\bX_2\in\widetilde\Theta_n$, write
\begin{align*}
&\frac{D(\prob_{\bX_1}||\prob_{\bX_2})}{\rho^2(\bX_1,\bX_2)}\\
&\quad = \frac{n^2}{\|\bX_1\bX_1\transpose - \bX_2\bX_2\transpose\|_{\mathrm{F}}^2}\sum_{i = 1}^{n-1}\sum_{j = i + 1}^{n}\left\{(u_{1i}u_{1j})\log\left(\frac{u_{1i}u_{1j}}{u_{2i}u_{2j}}\right) + (1 - u_{1i}u_{1j})\log\left(\frac{1 - u_{1i}u_{1j}}{1 - u_{2i}u_{2j}}\right)\right\}\\
&\quad\leq \frac{n^2}{\|\bX_1\bX_1\transpose - \bX_2\bX_2\transpose\|_{\mathrm{F}}^2}\sum_{i = 1}^{n-1}\sum_{j = i + 1}^{n}\left\{(u_{1i}u_{1j})\left(\frac{u_{1i}u_{1j}}{u_{2i}u_{2j}} - 1\right) + (1 - u_{1i}u_{1j})\left(\frac{1 - u_{1i}u_{1j}}{1 - u_{2i}u_{2j}} - 1\right)\right\}\\
&\quad\leq \frac{n^2}{\|\bX_1\bX_1\transpose - \bX_2\bX_2\transpose\|_{\mathrm{F}}^2}\sum_{i = 1}^{n}\sum_{j = 1}^{n}\frac{\left(u_{1i}u_{1j} - u_{2i}u_{2j}\right)^2}{u_{2i}u_{2j}(1 - u_{2i}u_{2j})}\\
&\quad\leq \frac{16d^2n^2}{\|\bX_1\bX_1\transpose - \bX_2\bX_2\transpose\|_{\mathrm{F}}^2}\sum_{i = 1}^n\sum_{j = 1}^n\left(u_{1i}u_{1j} - u_{2i}u_{2j}\right)^2 = 16d^2n^2.
\end{align*}
Hence $A\leq 16d^2n^2$. Applying Lemma \ref{lemma:Fano_lemma} and the covering number estimate \eqref{eqn:covering_number_RDPG} yields that
\[
\inf_{\widehat\bX}\sup_{\bX\in\widetilde\Theta_n}\expect_\bX\left\{\frac{1}{n^2}\|\widehat\bX\widehat\bX\transpose - \bX\bX\transpose\|_{\mathrm{F}}^2\right\}\gtrsim\frac{1}{n}.
\]
Finally, observe that for any $\bW\in\mathbb{O}(d)$,
\begin{align*}
\|\widehat\bX\widehat\bX\transpose - \bX\bX\transpose\|_{\mathrm{F}}
&\leq \|(\widehat\bX - \bX\bW)\widehat\bX\transpose\|_{\mathrm{F}} + \|(\bX\bW)(\widehat\bX - \bX\bW)\transpose\|_{\mathrm{F}}\\
&\leq \|\widehat\bX - \bX\bW\|_{\mathrm{F}}(\|\widehat\bX\|_{\mathrm{F}} + \|\bX\|_{\mathrm{F}})\lesssim \sqrt{n}\|\widehat\bX - \bX\bW\|_{\mathrm{F}}
\end{align*}
by the assumption that $\|\widehat\bX\|_{\mathrm{F}}\lesssim \sqrt{n}$ with probability one. 
Namely, 
\[
\inf_{\bW\in\mathbb{O}(d)}\|\widehat\bX - \bX\bW\|_{\mathrm{F}}\gtrsim (1/\sqrt{n})\|\widehat\bX\widehat\bX\transpose - \bX\bX\transpose\|_{\mathrm{F}},\]
and this implies that
\[
\inf_{\widehat\bX}\sup_{\bX\in\widetilde\Theta_n}\expect_\bX\left\{\frac{1}{n}\inf_{\bW\in\mathbb{O}(d)}\|\widehat\bX - \bX\bW\|_{\mathrm{F}}^2\right\}\gtrsim \inf_{\widehat\bX}\sup_{\bX\in\widetilde\Theta_n}\expect_\bX\left\{\frac{1}{n^2}\|\widehat\bX\widehat\bX\transpose - \bX\bX\transpose\|_{\mathrm{F}}^2\right\}\gtrsim\frac{1}{n},
\]
completing the proof. 
\end{proof}


\section{Proofs for Section \ref{sec:optimal_bayesian_estimation}} 
\label{sec:proofs_for_section_sub:likelihood_based_bayesian_estimation}

The blueprint of the proof of Theorem \ref{thm:root_n_contraction_BRDPG} can be described as a ``prior-mass-and-testing'' technique originally presented in the seminal work \cite{ghosal2000convergence}. Roughly speaking, the ``prior-mass'' technique is to show the denominator
\[
D_n = \int_{\calX^n} \prod_{i\leq j}\frac{p(y_{ij}\mid \bX)}{p(y_{ij}\mid\bX_0)}\Pi(\mathrm{d}\bX)
\]
appearing in the posterior distribution can be bounded from below with large probability, and the ``testing'' technique requires the construction of suitable test functions. In what follows we formalize these steps.

\subsection{Bounding the denominator $D_n$ from below} 
\label{sub:bounding_the_denominator_}

\begin{lemma}\label{lemma:evidence_LB}
Let $\bY\sim\mathrm{RDPG}(\bX_0)$ for some $\bX_0\in\calX^n$.
Assume that $\delta\leq\min_{i,j}\bx_{0i}\transpose\bx_{0j}\leq \max_{i,j}\bx_{0i}\transpose\bx_{0j}\leq 1-\delta$ for some constant $\delta\in(0, 1/2)$ independent of $n$, and that $\pi_\bx$ is bounded away froom $0$ and $\infty$. Then for any constants $\beta,\gamma > 0$, and for sufficiently large $n$ and sufficiently small $\eps > 0$,
\begin{align*}
\prob_0\left[
D_n\leq\exp\left\{-\left(\frac{16\beta^2}{\delta^2} + c_\pi + d\log\beta\right)n - \gamma n\eps^2 - nd\left(\log \frac{1}{\eps}\right)\right\}
\right]\leq \exp\left(-\frac{\gamma^2\delta^2n^2\eps^2}{128\beta^2}\right)
\end{align*}
for some constant $c_\pi$ independent of $n$ and $d$. 
\end{lemma}
\begin{proof}
For any constant $\beta > 0$, set $\calE_n = \left\{\bX:\|\bX - \bX_0\|_{2\to\infty} < {\beta}\eps\right\}$. Denote $\bP = \bX\bX\transpose = [\bP_{ij}]_{n\times n}$ and $\bP_0 = \bX_0\bX_0\transpose = [\bP_{0ij}]_{n\times n}$. 
Write
\begin{align*}
D_n & \geq \Pi(\calE_n)\int_{\calE_n}\prod_{i\leq j}\frac{p(y_{ij}\mid\bX)}{p(y_{ij}\mid\bX_0)}\Pi(\mathrm{d}\bX\mid\calE_n)\\
& = \Pi(\calE_n)\int_{\calE_n}\exp\left[
\sum_{i = 1}^n\sum_{j = i}^n\left\{-y_{ij}\log\left(\frac{\bP_{0ij}}{\bP_{ij}}\right) - (1 - y_{ij})\log\left(\frac{1 - \bP_{0ij}}{1 - \bP_{ij}}\right)
\right\}\right]\Pi(\mathrm{d}\bX\mid\calE_n)\\
&\geq \Pi(\calE_n)\int_{\calE_n}\exp\left[
\sum_{i = 1}^n\sum_{j = i}^n\left\{-y_{ij}\left(\frac{\bP_{0ij} - \bP_{ij}}{\bP_{ij}}\right) -  (1 - y_{ij})\left(\frac{\bP_{ij} - \bP_{0ij}}{1 - \bP_{ij}}\right)
\right\}\right]\Pi(\mathrm{d}\bX\mid\calE_n)\\
&= \Pi(\calE_n)\int_{\calE_n}\exp\left[-
\sum_{i = 1}^n\sum_{j = i}^n\left\{(y_{ij} - \bP_{0ij})\left(\frac{\bP_{0ij} - \bP_{ij}}{\bP_{ij}(1 - \bP_{ij})}\right) + \frac{\left(\bP_{0ij} - \bP_{ij}\right)^2}{\bP_{ij}(1 - \bP_{ij})}
\right\}\right]\Pi(\mathrm{d}\bX\mid\calE_n)\\
&\geq \Pi(\calE_n)\exp\left\{-
\sum_{i = 1}^n\sum_{j = i}^n(y_{ij} - \bP_{0ij})\int_{\calE_n}\left(\frac{\bP_{ij} - \bP_{0ij}}{\bP_{ij}(1 - \bP_{ij})}\right)\Pi(\mathrm{d}\bX\mid\calE_n)
\right\}\\
&\quad\times\exp\left\{-\sum_{i = 1}^n\sum_{j = i}^n\int_{\calE_n}\frac{\left(\bP_{0ij} - \bP_{ij}\right)^2}{\bP_{ij}(1 - \bP_{ij})}\Pi(\mathrm{d}\bX\mid\calE_n)\right\},
\end{align*}
where the third line follows from the fact that $\log x\leq (x - 1)$ for all $x > 0$, and the last inequality is due to Jensen's inequality. 
Since for any $\bX\in\calE_n$, we have, for any $i,j\in[n]$,
\begin{align*}
|\bP_{ij} - \bP_{0ij}|&\leq |(\bx_i - \bx_{0i})\transpose\bx_j| + |\bx_{0i}\transpose(\bx_j - \bx_{0j})|\leq (\|\bx_j\|_2 + \|\bx_{0i}\|_2)\|\bX - \bX_0\|_{2\to\infty}\leq \frac{\delta}{2},\\
\bP_{ij}(1 - \bP_{ij})&\geq (\bx_{0i}\transpose\bx_{0j} - |\bP_{ij} - \bP_{0ij}|)(1 - \bx_{0i}\transpose\bx_{0j} - |\bP_{ij} - \bP_{0ij}|)\geq\frac{\delta^2}{4},
\\
\left|\frac{\bP_{ij} - \bP_{0ij}}{\bP_{ij}(1 - \bP_{ij})}\right|
&\leq \max_{i,j\in[n]}\frac{|(\bx_i - \bx_{0i})\transpose\bx_j| + |\bx_{0i}\transpose(\bx_j - \bx_{0j})|}{\bP_{ij}(1 - \bP_{ij})}
\leq \frac{2\|\bX - \bX_0\|_{2\to\infty}}{\delta^2/4}\leq \frac{8\beta\eps}{\delta^2},\\
\frac{(\bP_{ij} - \bP_{0ij})^2}{\bP_{ij}(1 - \bP_{ij})}
&\leq \max_{i,j\in[n]}\frac{2|(\bx_i - \bx_{0i})\transpose\bx_j|^2 + 2|\bx_{0i}\transpose(\bx_j - \bx_{0j})|^2}{\bP_{ij}(1 - \bP_{ij})}
\leq \frac{4\|\bX - \bX_0\|_{2\to\infty}^2}{\delta^2/4}\leq \frac{16\beta^2\eps^2}{\delta^2},
\end{align*}
implying that
\begin{align*}
\left|\int_{\calE_n}\left(\frac{\bP_{ij} - \bP_{0ij}}{\bP_{ij}(1 - \bP_{ij})}\right)\Pi(\mathrm{d}\bX\mid\calE_n)\right|
\leq \frac{8\beta\eps}{\delta^2},\quad
\int_{\calE_n}\frac{(\bP_{ij} - \bP_{0ij})^2}{\bP_{ij}(1 - \bP_{ij})}\Pi(\mathrm{d}\bX\mid\calE_n)
\leq \frac{16\beta^2\eps^2}{\delta^2},
\end{align*}
it follows that for any $\gamma > 0$,
\begin{align*}
&\left\{D_n\leq\Pi(\calE_n)\exp\left\{-\frac{n(n + 1)}{2}\left(\frac{16\beta^2\eps^2}{\delta^2}\right) - \gamma n^2\eps^2\right\}\right\}\\
&\quad\subset
\left\{
\sum_{i = 1}^n\sum_{j = i}^n(y_{ij} - \bP_{0ij})\int_{\calE_n}\left(\frac{\bP_{ij} - \bP_{0ij}}{\bP_{0ij}(1 - \bP_{0ij})}\right)\Pi(\mathrm{d}\bX\mid\calE_n) > \gamma n^2\eps^2
\right\}.
\end{align*}
Hence, by Hoeffding's inequality,
\begin{align*}
&\prob_0\left[D_n\leq\Pi(\calE_n)\exp\left\{-\frac{n(n + 1)}{2}\left(\frac{16\beta^2\eps^2}{\delta^2}\right) - \gamma n^2\eps^2\right\}\right]\\
&\quad\leq \prob_0\left\{
\sum_{i = 1}^n\sum_{j = i}^n(y_{ij} - \bP_{0ij})\int_{\calE_n}\left(\frac{\bP_{ij} - \bP_{0ij}}{\bP_{0ij}(1 - \bP_{0ij})}\right)\Pi(\mathrm{d}\bX\mid\calE_n) > \gamma n^2\eps^2
\right\}\\
&\quad\leq\exp\left\{-\frac{4\gamma^2 n^4\eps^4}{n(n + 1)}\left(\frac{\delta^2}{16\beta\eps}\right)^2\right\} = \exp\left(-\frac{\gamma^2\delta^4n^2\eps^2}{128\beta^2}\right).
\end{align*}
It suffices to provide an exponential lower bound for $\Pi(\calE_n)$. This can be easily obtained using the fact that $\pi_\bx(\bx_i)\mathrm{vol}(B_1^d)\geq \exp(c_\pi) > 0$ for some constant $c_\pi$: for sufficiently small $\eps$, $\{\|\bx_i - \bx_{0i}\|_2 < \beta\eps\}\subset\calX$, and thus,
\begin{align}
\Pi(\calE_n)& = \prod_{i = 1}^n\int_{\{\|\bx_i - \bx_{0i}\|_2 < \beta/\sqrt{n}\}}\pi_\bx(\bx_i)\mathrm{d}\bx_i
 \geq \prod_{i = 1}^n\left\{\exp(c_\pi)\frac{\mathrm{vol}\left(\{\bx:\|\bx - \bx_{0i}\|_2 < \beta\eps\}\right)}{\mathrm{vol}(B_1^d)}\right\}\nonumber\\
 \label{eqn:prior_concentration}
& = \exp\left(nc_\pi\right)\left({\beta}\eps\right)^{nd}
 = \exp\left\{n(c_\pi + d\log\beta) - nd\left(\log\frac{1}{\eps}\right)\right\}.
\end{align}
Namely, for any $\beta,\gamma > 0$, we obtain the following conclusion,
\begin{align*}
\prob_0\left[
D_n\leq\exp\left\{-\left(\frac{16\beta^2}{\delta^2} + c_\pi + d\log\beta\right) n -\gamma n^2\eps^2 - nd\left(\log \frac{1}{\eps}\right)\right\}
\right]\leq \exp\left(-\frac{\gamma^2\delta^4n^2\eps^2}{128\beta^2}\right),
\end{align*}
where $c_\pi$ is some constant depending independent of $d$ and $n$. The proof is thus completed.
\end{proof}

\subsection{Construction of test functions} 
\label{sub:construction_of_test_functions}


\begin{lemma}\label{lemma:entropy_prior_mass}
Let $M > 0$, consider the pseudo-metric $\rho(\bX,\bX_0) = \|\bX\bX\transpose - \bX_0\bX_0\transpose\|_{\mathrm{F}}/n$, and take 
$\Theta_n = \left\{\bX\in\calX^n:\rho(\bX, \bX_0)\leq M\sqrt{{(d\log n)}{n}}\right\}$.
Assume that $\sigma_d(\bX_0)\geq\sigma_0\sqrt{n/d}$ for some constant $\sigma_0 > 0$ that is independent of $n$ and $d$. 
If $(d^4\log n)/ n\to0$ and $\pi_\bx$ is bounded away from $0$ and $\infty$, then the following inequalities hold for sufficiently large $n$:
\begin{align*}
\calN\left(\frac{\eps}{4}, \left\{\bX\in\Theta_n:\rho(\bX,\bX_0) < \eps\right\}, \rho\right)
&\leq \left(\frac{3}{\eps}\right)^{d^2}\left\{
24\sqrt{d}\left(1 + \frac{16}{\sigma_0} + \frac{8}{\sigma_0^2}\right)\right\}^{n d}
\end{align*}
for all $\eps > 0$, and
\begin{align*}
\Pi\left(\bX\in\Theta_n:\rho(\bX,\bX_0)\leq 2j\eps\right)
&\leq\left(\frac{3}{\eps}\right)^{d^2}\exp\left\{(C_\pi - \log\mathrm{vol}(B_1^d))n\right\}\left(\sqrt{2\pi\mathrm{e}}Cj\eps\right)^{nd}
\end{align*}
for all sufficiently small $\eps > 0$.
\end{lemma}

\begin{proof}
Denote $\calF = \{\bX\in\Theta_n:\rho(\bX,\bX_0) < \eps\}$. We first show that for any $\bX\in\Theta_n$, $\sigma_d(\bX) \geq \sigma_0\sqrt{n/d}/2$ for sufficiently large $n$. For any $\bX\in\calF$, by the Weyl's inequality, we have, for sufficiently large $n$, 
\begin{align*}
|\sigma_d(\bX) - \sigma_d(\bX_0)|
&=\frac{|\lambda_d(\bX\bX\transpose{}) - \lambda_d(\bX_0\bX_0\transpose{})|}{\sigma_d(\bX) + \sigma_d(\bX_0)}
\leq\sqrt{\frac{d}{\sigma_0^2n}}\|\bX\bX\transpose - \bX_0\bX_0\transpose\|_{\mathrm{F}}
\\&
=\sqrt{\frac{nd}{\sigma_0^2}}\rho(\bX, \bX_0)\leq\frac{\sigma_0}{2}\sqrt{\frac{n}{d}},
 \end{align*}
and hence,
$\sigma_d(\bX) \geq \sigma_d(\bX_0) - |\sigma_d(\bX) - \sigma_d(\bX_0)|\geq {\sigma_0}\sqrt{{n}/{d}}/2$.

Let $\bX_1,\bX_2\in\Theta_n$, and let them yield singular value decompositions $\bX_1 = \bU_1\bS_1^{1/2}\bV_1\transpose$ and $\bX_2 = \bU_2\bS_2^{1/2}\bV_2\transpose$, where $\bU_1,\bU_2\in\mathbb{O}(n, d)$ and $\bV_1,\bV_2\in\mathbb{O}(d)$. Further let $\bU_2\transpose\bU_1 = \bW_1\bSigma\bW_2\transpose$ be the singular value decomposition of $\bU_2\transpose\bU_1$, and let $\bW_\bU = \bW_1\bW_2\transpose$. 
Denote $\bP_1 = \bX_1\bX_1\transpose$ and $\bP_2 = \bX_2\bX_2\transpose$. Then by Lemma \ref{lemma:Decomposition} and the fact that $\|\bA\bB\|_{\mathrm{F}} \leq \|\bA\|_{\mathrm{F}}\|\bB\|_2$, we have,
\begin{align*}
\|\bX_1\bV_1 - \bX_2\bV_2\bW_\bU\|_{\mathrm{F}}
&\leq \|\bS_2^{-1/2}\|_2\|\bP_1 - \bP_2\|_{\mathrm{F}} + 
\|\bP_1 - \bP_2\|_{\mathrm{F}}(\|\bS_1^{-1/2}\|_2 + \|\bS_2^{-1/2}\|)
\\&\quad
 + \|\bP_1 - \bP_2\|_{\mathrm{F}}\|\bS_1^{-1/2}\|_2
 + 2\|\bP_1 - \bP_2\|_{\mathrm{F}}\|\bS_1^{-1/2}\|_2 + \|\sin\Theta(\bU_1,\bU_2)\|_2^2\|\bS_1^{1/2}\|_{\mathrm{F}}
 \\&\quad
  + \|\bW_\bU\bS_1^{1/2} - \bS_2^{1/2}\bW_\bU\|_{\mathrm{F}}\\
&\leq \frac{12}{\sigma_0}\sqrt{\frac{d}{n}}\|\bP_1 - \bP_2\|_{\mathrm{F}} + \sqrt{n}\|\sin\Theta(\bU_1,\bU_2)\|_2^2 + \|\bW_\bU\bS_1^{1/2} - \bS_2^{1/2}\bW_\bU\|_{\mathrm{F}}.
\end{align*}
By Davis-Kahan theorem, we have, for sufficiently large $n$,
\begin{align*}
\|\sin\Theta(\bU_1,\bU_2)\|_2
&\leq
 \frac{8d\|\bP_1 - \bP_2\|_{\mathrm{F}}}{\sigma_0^2n}\leq\frac{8d}{\sigma_0^2n}\left\{
n\rho(\bX_1,\bX_0) + n\rho(\bX_2,\bX_0)
\right\}
\leq \frac{8M}{\sigma_0^2}\sqrt{\frac{d^3\log n}{n}}\leq\frac{1}{\sqrt{d}},
\end{align*}
as we are assuming that $(d^4\log n)/n\to 0$. 
Therefore,
\begin{align*}
\sqrt{n}\|\sin\Theta(\bU_1,\bU_2)\|_2^2\leq \frac{4d}{\sigma_0^2\sqrt{n}}\|\sin\Theta(\bU_1,\bU_2)\|_2\|\bP_1 - \bP_2\|_{\mathrm{F}}\leq \frac{4}{\sigma_0^2}\sqrt{\frac{d}{n}}\|\bP_1 - \bP_2\|_{\mathrm{F}}.
\end{align*}
To tackle the last term $\|\bW_\bU\bS_1^{1/2} - \bS_2^{1/2}\bW_\bU\|_{\mathrm{F}}$, we adopt the technique applied in \cite{JMLR:v18:17-448} (see Lemma 2 there) and derive the following decomposition of $\bW_\bU\bS_1 - \bS_2 \bW_\bU$:
\begin{align*}
\bW_\bU\bS_1 - \bS_2 \bW_\bU & = (\bW_\bU - \bU_2\transpose\bU_1)\bS_1 + \bU_2\transpose\bU_1\bS_1 - \bS_2\bW_\bU\\
& = (\bW_\bU - \bU_2\transpose\bU_1)\bS_1 +\bU_2\transpose\bP_1\bU_1 - \bS_2\bW_\bU\\
& = (\bW_\bU - \bU_2\transpose\bU_1)\bS_1 +\bU_2\transpose(\bP_1 -\bP_2)\bU_1 + \bU_2\transpose\bP_2\bU_1 - \bS_2\bW_\bU\\
& = (\bW_\bU - \bU_2\transpose\bU_1)\bS_1 +\bU_2\transpose(\bP_1 -\bP_2)\bU_1 + \bS_2\bU_2\transpose\bU_1 - \bS_2\bW_\bU\\
& = (\bW_\bU - \bU_2\transpose\bU_1)\bS_1 +\bU_2\transpose(\bP_1 -\bP_2)\bU_1 + \bS_2(\bU_2\transpose\bU_1 - \bW_\bU).
\end{align*}
Since by Davis-Kahan theorem,
\[
\|\bW_\bU - \bU_2\transpose\bU_1\|_2\leq \|\sin\Theta(\bU_1,\bU_2)\|_2^2\leq \frac{8\sqrt{d}}{\sigma_0^2n}\|\bP_1 - \bP_2\|_{\mathrm{F}},\]
it follows that
\begin{align*}
\|\bW_\bU\bS_1 - \bS_2 \bW_\bU\|_{\mathrm{F}}
&\leq \|\bW_\bU - \bU_2\transpose\bU_1\|_2(\|\bS_1\|_{\mathrm{F}} + \|\bS_2\|_{\mathrm{F}}) + \|\bP_1 - \bP_2\|_{\mathrm{F}}\leq 2\|\bP_1 - \bP_2\|_{\mathrm{F}}.
\end{align*}
Observe that the $(i,j)$-th entry of $\bW_\bU\bS_1^{1/2} - \bS_2^{1/2} \bW_\bU$ can be written as
\[
(\bW_\bU)_{ij}\left\{\sqrt{\lambda_j(\bP_1)} - \sqrt{\lambda_i(\bP_2)}\right\} = \frac{(\bW_\bU)_{ij}\left\{(\lambda_j(\bP_1) - \lambda_i(\bP_2)\right\}}{\sqrt{\lambda_j(\bP_1)} + \sqrt{\lambda_i(\bP_2)}} = \frac{(\bW_\bU\bS_1 - \bS_2 \bW_\bU)_{ij}}{\sqrt{\lambda_j(\bP_1)} + \sqrt{\lambda_i(\bP_2)}},
\]
it follows that
\begin{align*}
\|\bW_\bU\bS_1^{1/2} - \bS_2^{1/2} \bW_\bU\|_{\mathrm{F}}\leq \frac{1}{\sigma_d(\bX_1)}\|\bW_\bU\bS_1 - \bS_2 \bW_\bU\|_{\mathrm{F}}\leq \frac{4}{\sigma_0}\sqrt{\frac{d}{n}}\|\bP_1 - \bP_2\|_{\mathrm{F}}.
\end{align*}
Combining the above results, we obtain, that for any $\bX_1,\bX_2\in\Theta_n$, 
\begin{align}\label{eqn:rho_eta_equivalence}
\|\bX_1 - \bX_2\bV_2\bW_\bU\bV_1\|_{\mathrm{F}}\leq \left(\frac{16}{\sigma_0} + \frac{8}{\sigma_0^2}\right)\sqrt{\frac{d}{n}} \|\bP_1 - \bP_2\|_{\mathrm{F}}.
\end{align}
Hence,
\begin{align*}
\calF &= \{\bX\in\Theta_n:\rho(\bX,\bX_0) < \eps\}\subset\left\{\bX\in\Theta_n:\inf_{\bW\in\mathbb{O}(d)}\|\bX - \bX_0\bW\|_{\mathrm{F}}\leq\left(\frac{16}{\sigma_0} + \frac{8}{\sigma_0^2}\right)\sqrt{nd}\eps\right\}\\
&=\bigcup_{\bW\in\mathbb{O}(d)}\left\{\bX\in\Theta_n:\|\bX - \bX_0\bW\|_{\mathrm{F}} < \left(\frac{16}{\sigma_0} + \frac{8}{\sigma_0^2}\right)\sqrt{nd}\eps\right\}.
\end{align*}
Now let $\widetilde\calO(\eps)$ be an $\eps/\sqrt{d}$-net of $(\mathbb{O}(d), \|\cdot\|_{\mathrm{F}})$. Since $\mathbb{O}(d)\subset\{\bA\in\mathbb{R}^{d\times d}:\|\bA\|_{\mathrm{F}} = 1\}$, it follows that $|\widetilde\calO(\eps)|\leq (3/\eps)^{d^2}$. Therefore, for any
\[
\bX \in \bigcup_{\bW\in\mathbb{O}(d)}\left\{\bX\in\Theta_n:\|\bX - \bX_0\bW\|_{\mathrm{F}} < \left(\frac{16}{\sigma_0} + \frac{8}{\sigma_0^2}\right)\sqrt{nd}\eps\right\},
\]
there exists some $\bW\in\mathbb{O}(d)$, and some $\bR\in\widetilde\calO(\eps)$, such that $\|\bW - \bR\|_{\mathrm{F}}\leq \eps$, and $\|\bX - \bX_0\bW\|_{\mathrm{F}} < \left({16}/{\sigma_0} + {8}/{\sigma_0^2}\right)\sqrt{nd}\eps$, implying that 
\[
\|\bX - \bX_0\bR\|_{\mathrm{F}}\leq \|\bX - \bX_0\bW\|_{\mathrm{F}} + \|\bX_0(\bW - \bR)\|_{\mathrm{F}} < \left(1 + \frac{16}{\sigma_0} + \frac{8}{\sigma_0^2}\right)\sqrt{nd}\eps.
\]
Hence,
\begin{align*}
\bX\in\bigcup_{\bR\in\widetilde\calO(\eps)}
\left\{\bX\in\Theta_n:\|\bX - \bX_0\bR\|_{\mathrm{F}} < \left(1 + \frac{16}{\sigma_0} + \frac{8}{\sigma_0^2}\right)\sqrt{nd}\eps\right\},
\end{align*}
and hence, by the fact that $\rho(\bX_1,\bX_2) \leq (2/\sqrt{n})\|\bX_1 - \bX_2\|_{\mathrm{F}}$,
\begin{align*}
\calN\left(\frac{\eps}{4}, \calF, \rho\right)
&\leq \sum_{\bR\in\widetilde\calO(\eps)}
\calN\left(\frac{\eps}{4}, \left\{\bX\in\Theta_n:\|\bX - \bX_0\bR\|_{\mathrm{F}} < \left(1 + \frac{16}{\sigma_0} + \frac{8}{\sigma_0^2}\right)\sqrt{nd}\eps\right\}, \rho\right)\\
&\leq\sum_{\bR\in\widetilde\calO(\eps)}
\calN\left(\frac{\eps}{4}, \left\{\bX\in\Theta_n:\|\bX - \bX_0\bR\|_{\mathrm{F}} < \left(1 + \frac{16}{\sigma_0} + \frac{8}{\sigma_0^2}\right)\sqrt{nd}\eps\right\}, \frac{2}{\sqrt{n}}\|\cdot\|_{\mathrm{F}}\right)\\
&\leq\left(\frac{3}{\eps}\right)^{d^2}\calN\left(\frac{\sqrt{n}\eps}{8}, \left\{\bX\in\mathbb{R}^{n\times d}:\|\bX - \bX_0\bR\|_{\mathrm{F}} < \left(1 + \frac{16}{\sigma_0} + \frac{8}{\sigma_0^2}\right)\sqrt{nd}\eps\right\}, \|\cdot\|_{\mathrm{F}}\right)\\
&\leq \left(\frac{3}{\eps}\right)^{d^2}\left\{
24\sqrt{d}\left(1 + \frac{16}{\sigma_0} + \frac{8}{\sigma_0^2}\right)\right\}^{n\times d},
\end{align*}
completing the proof of the first assertion. For the second assertion, we proceed similarly to derive
\begin{align*}
\left\{\bX\in\Theta_n:\rho(\bX,\bX_0) \leq 2j\eps\right\}
&\subset\bigcup_{\bW\in\mathbb{O}(d)}\left\{\bX\in\Theta_n:\|\bX - \bX_0\bW\|_{\mathrm{F}} \leq \left(\frac{16}{\sigma_0} + \frac{8}{\sigma_0^2}\right)2\sqrt{nd}j\eps\right\}\\
& = \bigcup_{\bW\in\widetilde\calO(\eps)}\left\{\bX\in\Theta_n:\|\bX - \bX_0\bW\|_{\mathrm{F}}\leq \left(1 + \frac{16}{\sigma_0} + \frac{8}{\sigma_0^2}\right)2\sqrt{nd}j\eps\right\}.
\end{align*}
Invoking the fact that $\mathrm{vol}(B_1^d)\sup_{\bx\in\calX}\pi_\bx\leq \exp(C_\pi)$ for some constant $C_\pi$ independent of $n$, we obtain, for some constant $C > 0$,
\begin{align*}
\Pi(\bX\in\Theta_n:\rho(\bX,\bX_0) \leq 2j\eps)
&\leq\sum_{\bW\in\widetilde\calO(\eps)}\Pi\left(\bX\in\Theta_n:\|\bX - \bX_0\bW\|_{\mathrm{F}}\leq C\sqrt{nd}j\eps\right)\\
& = \sum_{\bW\in\widetilde\calO(\eps)}\int\ldots\int_{\left\{\|\bX - \bX_0\bW\|_{\mathrm{F}}\leq C\sqrt{nd}j\eps\right\}}\prod_{i = 1}^n\pi_\bx(\bx_i)\mathrm{d}\bx_1\ldots\mathrm{d}\bx_n\\
& \leq \left(\frac{3}{\eps}\right)^{d^2}\frac{\mathrm{vol}(B_1^{nd})}{\mathrm{vol}(B_1^d)^n}\exp(nC_\pi)\left(C\sqrt{nd}j\eps\right)^{nd}\\
& \leq \left(\frac{3}{\eps}\right)^{d^2}\exp\left\{(C_\pi - \log\mathrm{vol}(B_1^d))n\right\}\left(\sqrt{2\pi\mathrm{e}}Cj\eps\right)^{nd}.
\end{align*}
The proof is thus completed. 
\end{proof}

\subsection{A coarse posterior contraction result for the edge probability matrix} 
\label{sub:a_coarse_posterior_contraction_result}


Theorem \ref{thm:root_n_contraction_BRDPG} in the manuscript claims that the posterior contraction rate is $1/n$ with respect to $(1/n)\inf_\bW\|\bX - \bX_0\bW\|_{\mathrm{F}}^2$. It turns out that it is easier to establish a coarser posterior contraction result with an extra logarithmic factor. We show in the following proposition that contraction rate for the edge probability matrix is $\sqrt{(\log n)/n}$ with respect to $(1/n)\|\bX\bX\transpose - \bX_0\bX_0\transpose\|_{\mathrm{F}}$. Note that Proposition \ref{prop:near_rootn_contraction_BRDPG} does not imply Theorem \ref{thm:root_n_contraction_BRDPG} but is a weaker result. 

To achieve this goal, we need the following local testing lemma tailored for random graph models, which was originally presented \cite{pati2015optimal}. 
\begin{lemma}[Lemma 4.2, \citealp{pati2015optimal}]
\label{lemma:local_test_random_graph}
Assume that $\bP_1,\bP_0\in[0, 1]^{n\times n}$ are two distinct edge probability matrices and let $\calE = \{\bP\in[0, 1]^{n\times n}:\|\bP - \bP_1\|_{\mathrm{F}}\leq \|\bP_1 - \bP_0\|_{\mathrm{F}}/2\}$ be a Frobenius ball of radius $\|\bP_1 - \bP_0\|_{\mathrm{F}}/2$ centered at $\bP_1$. Based on $Y_{ij}\sim\mathrm{Bernoulli}(\bP_{ij})$ for $1\leq i\leq j\leq n$, consider testing $H_0:\bP = \bP_0$ against $H_A:\bP\in\calE$. Then there exists a test function $\phi_n$, such that
\[
\expect_{\bP_0}(\phi_n)\leq\exp(-C_1\|\bP_1 - \bP_0\|_{\mathrm{F}}^2),\quad
\sup_{\bP\in\calE}\expect_\bP(1 - \phi_n)\leq\exp(-C_2\|\bP_1 - \bP_0\|_{\mathrm{F}}^2)
\]
for some universal constants $C_1, C_2 > 0$ independent of $\bP_0,\bP_1$, and $n$. 
\end{lemma}

\begin{proposition}\label{prop:near_rootn_contraction_BRDPG}
Under the assumption of Theorem \ref{thm:root_n_contraction_BRDPG}, there exists some absolute 
constant $K > 0$ and some large constant $M > 0$, such that 
\begin{align*}
\expect_0\left\{\Pi\left(\frac{1}{n}\|\bX\bX\transpose - \bX_0\bX_0\transpose\|_{\mathrm{F}} > M\sqrt{\frac{d\log n}{n}}\mathrel{\bigg|}\bY\right)\right\}\leq 3\exp\left(-\frac{1}{2}nd\log n\right)
\end{align*}
for sufficiently large $n$. 
\end{proposition}
\begin{proof}
Take $\eps_n = \sqrt{(d\log n)/n}$. Let $\beta,\gamma > 0$ be constants to be determined later. Denote the event
\[
\Xi_n = 
\left\{\bY:
D_n>\exp\left\{-\left(\frac{16\beta^2}{\delta^2} + c_\pi + d\log\beta\right) - \gamma nd \log n - nd\left(\log \frac{1}{\eps_n}\right)\right\}
\right\}.
\]
Consider the pseudo-metric $\rho:\calX^n\times\calX^n\to[0,\infty)$ defined by $\rho(\bX_1,\bX_2) = (1/n)\|\bX_1\bX_1\transpose - \bX_2\bX_2\transpose\|_{\mathrm{F}}$. Let $\{\bX_1,\ldots,\bX_s\}$ be an $\eps_n/2$-net of $(\{\bX\in\calX^n:\rho(\bX,\bX_0) > M\eps_n\}, \rho)$. Clearly, 
\[
\rho(\bX_1,\bX_2)\leq \frac{1}{n}\left\{\|\bX_1(\bX_1 - \bX_2)\transpose\|_{\mathrm{F}} + \|(\bX_1 - \bX_2)\bX_2\transpose\|_{\mathrm{F}}\right\}\leq \frac{2}{\sqrt{n}}\|\bX_1 - \bX_2\|_{\mathrm{F}},
\]
implying that
\begin{align*}
s&\leq \calN\left(\frac{\eps_n}{2}, \calX^n, \rho\right)\leq \calN\left(\frac{\eps_n}{2}, \calX^n, \frac{2}{\sqrt{n}}\|\cdot\|_{\mathrm{F}}\right)\\
& \leq \calN\left(\frac{\sqrt{n}\eps_n}{4}, \{\bX\in\mathbb{R}^{n\times d}:\|\bX\|_{\mathrm{F}}\leq \sqrt{n}\}, \|\cdot\|_{\mathrm{F}}\right)\leq\left(\frac{12}{\eps_n}\right)^{nd}.
\end{align*}
For each $r = 1,\ldots,s$, it can be seen that $\bX\in B_\rho(\bX_r, \eps_n/2)$
 implies that $\rho(\bX , \bX_r)<\eps_n/2\leq M\eps_n/2\leq\rho(\bX_r, \bX_0)/2$.
This allows us to invoke Lemma \ref{lemma:local_test_random_graph} to construct test functions $\phi_{rn}$, $r\in[s]$, such that
\begin{align*}
\expect_0\phi_{rn}&\leq \exp\left\{-C_1n^2\rho^2(\bX_r, \bX_0)\right\}\leq \exp\left\{-Kn^2M^2\eps_n^2\right\} = \exp(-KM^2nd\log n),\\
\sup_{\bX\in B_\rho(\bX_r, \eps_n/2)}\expect_{\bX}\left(1 - \phi_{rn}\right)&\leq \exp\left\{-C_2n^2\rho^2(\bX_r, \bX_0)\right\}\leq \exp\left\{-Kn^2M^2\eps_n^2\right\}= \exp(-KM^2nd\log n)
\end{align*}
for some constant $K = \min\{C_1, C_2\}$. Taking $\phi_n = \max_{r\in[s]}\phi_{rn}$ yields the following bounds for the type I and type II error probabilities:
\begin{align}
\expect_0\phi_n &= \expect_0\left(\max_{r \in [s]}\phi_{rn}\right) \leq \sum_{r = 1}^s\expect_0(\phi_{rn})\nonumber\\
&\leq \exp\left\{nd\log 12 + nd\left(\log\frac{1}{\eps_n}\right) - KM^2nd\log n\right\}\nonumber
\\&
\label{eqn:typeI_error_test}
\leq \exp\left\{- \left(KM^2 - 3\right)nd\log n\right\},\\
\sup_{\bX:\rho(\bX,\bX_0) > M\eps_n}\expect_\bX(1 - \phi_n)& \leq  
\max_{r\in [s]}\sup_{\bX \in B_\rho(\bX_r,\eps_n/2)}\expect_\bX\left(1 - \max_{r\in[s]}\phi_{rn}\right)
\nonumber\\
\label{eqn:typeII_error_test}
&\leq \max_{r\in[s]}\sup_{\bX \in B_\rho(\bX_r,\eps_n/2)}\expect_\bX\left(1 - \phi_{rn}\right)\leq \exp(-KM^2nd\log n).
\end{align}
We are now in a position to provide an exponential upper bound for $\expect_0\left[\Pi\{\rho(\bX,\bX_0) > M\eps_n\mid\bY\}\right]$:
\begin{align*}
&\expect_0\left[\Pi\{\rho(\bX,\bX_0) > M\eps_n\mid\bY\}\right]\\
&\quad \leq \expect_0\left\{\frac{N_n(\{\bX:\rho(\bX,\bX_0) > M\eps_n\})}{D_n}\mathbbm{1}(\bY\in\Xi_n)(1 - \phi_n)\right\} + \expect_0(\phi_n) + \prob_0(\Xi_n^c)\\
&\quad \leq \max_{\bY\in\Xi_n}\left(\frac{1}{D_n}\right)\expect_0\left\{(1 - \phi_n)\int_{\{\bX:\rho(\bX,\bX_0) > M\eps_n\}}\frac{p(\bY\mid\bX)}{p(\bY\mid\bX_0)}\pi_\bX(\bX)\mathrm{d}\bX\right\} + \expect_0(\phi_n) + \prob_0(\Xi_n^c).
\end{align*}
By Fubini's theorem and inequality \eqref{eqn:typeII_error_test}, the expected value appearing in the first term of the right-hand side of the above display can be further upper bounded:
\begin{align*}
&\expect_0\left\{(1 - \phi_n)\int_{\{\bX:\rho(\bX,\bX_0) > M\eps_n\}}\frac{p(\bY\mid\bX)}{p(\bY\mid\bX_0)}\pi_\bX(\bX)\mathrm{d}\bX\right\}
\\&\quad
 = \int_{\{\bX:\rho(\bX,\bX_0) > M\eps_n\}}\expect_0\left[(1 - \phi_n)\frac{p(\bY\mid\bX)}{p(\bY\mid\bX_0)}\right]\pi_\bX(\bX)\mathrm{d}\bX
 \\&\quad
 \leq \int\sup_{\{\bX:\rho(\bX,\bX_0) > M\eps_n\}}\expect_\bX\left[(1 - \phi_n)\right]\pi_\bX(\bX)\mathrm{d}\bX\leq \exp(-KM^2nd\log n).
\end{align*}
Hence, invoking Lemma \ref{lemma:evidence_LB} and inequality \eqref{eqn:typeI_error_test} and setting $\beta = \delta$, $\gamma = 8$, we have, for some constant $c(\delta)$ depending only on $\delta$, that
\begin{align*}
\expect_0\left[\Pi\{\rho(\bX,\bX_0) > M\eps_n\mid\bY\}\right]&
\leq \exp\left\{(16 + c_\pi + d\log\delta)n + 9nd\log n - KM^2nd\log n\right\}\\
&\quad + \exp\left\{-\left(KM^2 - 3\right)nd\log n\right\} + \exp\left(-\frac{nd\log n}{2}\right)\\
&\leq 2\exp\left\{-\left(KM^2 - c(\delta)\right)nd\log n\right\} + \exp\left(-\frac{nd\log n}{2}\right).
\end{align*}
Taking $M$ sufficiently large such that $KM^2 - c(\delta) > 1/2$ completes the proof. 
\end{proof}

\subsection{Refinement of posterior contraction by restriction} 
\label{sub:refinement_of_posterior_contraction_by_restriction}


We now refine the contraction rate in Proposition \ref{prop:near_rootn_contraction_BRDPG} but the restriction of the posterior distribution over the set $\{(1/n)\|\bX\bX\transpose - \bX_0\bX_0\transpose\|_{\mathrm{F}}\leq M\sqrt{(\log n)/n}\}$. This will require the use of the following global testing Lemma, which was originally presented in \cite{ghosal2007convergence}, and is adapted to the random dot product graph model for our purpose. 

\begin{lemma}[Lemma 9, \citealp{ghosal2007convergence}]
\label{lemma:global_test}
Let $\bY\sim\mathrm{RDPG}(\bX)$ for some $\bX\in\calX^n$.
Define a pseudo-metric $\rho:\mathbb{R}^{n\times d}\times\mathbb{R}^{n\times d}\to [0,\infty)$ by $\rho(\bX, \bX_0) = (1/n)\|\bX\bX\transpose - \bX_0\bX_0\transpose\|_{\mathrm{F}}$. Let $\Theta_n\subset\calX^n$ be a collection of $n\times d$ matrices that lie in $\calX^n$. Suppose that for some non-increasing function $\eps\mapsto N(\eps)$, and some $\eps_n\geq0$,
\[
\calN\left(\frac{\eps}{4}, \left\{\bX\in\Theta_n:\rho(\bX,\bX_0) < \eps\right\}, \rho\right)\leq N(\eps)\quad\text{for all }\eps > \eps_n.
\]
Then for $K = \min\{C_1,C_2\}$ appearing in Lemma \ref{lemma:local_test_random_graph}, and for any $\eps > \eps_n$, there exists a test function $\phi_n$ for testing testing $H_0:\bX = \bX_0$ versus $H_A:\bX\in\Theta_n,\rho(\bX, \bX_0) > j\eps$ that depends on $\eps$, such that for every $j\in\mathbb{N}$, 
\begin{align*}
\expect_0(\phi_n)\leq N(\eps)\frac{\exp(-Kn^2\eps^2)}{1 - \exp(-Kn^2\eps^2)},\qquad
\sup_{\bX\in\Theta_n,\rho(\bX,\bX_0)>j\eps}\expect_\bX(1 - \phi_n)\leq \exp(-Kn^2\eps^2j^2).
\end{align*}
\end{lemma}

\begin{proof}[of Theorem \ref{thm:root_n_contraction_BRDPG}]
Denote the target posterior contraction rate $\eps_n = n^{-1/2}$. 
Consider the pseudo-metric $\rho(\bX_1,\bX_2) = (1/n)\|\bX_1\bX_1\transpose - \bX_2\bX_2\transpose\|_{\mathrm{F}}$. By Proposition \ref{prop:near_rootn_contraction_BRDPG}, we can take $\Theta_n = \{\bX\in\calX^n:\rho(\bX,\bX_0) < M\sqrt{(d\log n)/n}\}$ for some large constant $M > 0$ such that
\[
\expect_0\{\Pi\left(\Theta_n\mid\bY\right)\}\leq 3\exp\left(-\frac{1}{2}nd\log n\right).
\] 
The proof is based on a refinement of that of Theorem 1 in \cite{ghosal2007convergence} with exponential error bound. We breakdown the proof into the following components.
\begin{itemize}
  \item \textbf{Component 1. }By the first assertion of Lemma \ref{lemma:entropy_prior_mass}, we have, for some constant $m > 0$ to be determined later,
  \begin{align*}
  &\sup_{\eps > m\eps_n}\log\calN\left(\frac{\eps}{4}, \{\bX\in\Theta_n:\rho(\bX,\bX_0) < \eps\}, \rho\right)\\
  &\quad\leq d^2\left(\log\frac{3}{m\eps_n}\right) + nd\log\left\{
  24\sqrt{d}\left(1 + \frac{16}{\sigma_0} + \frac{8}{\sigma_0^2}\right)\right\}
  \leq Ln^2\eps_n^2
  \end{align*}
  for some constants $L > 0$. Invoking Lemma \ref{lemma:global_test}, we obtain some test function $\phi_n$, such that for any $j\in\mathbb{N}_+$, 
  \begin{align*}
  \expect_0(\phi_n)\leq \frac{\exp\{- (Km^2 - L)n\}}{1 - \exp(-Km^2n)},\quad
  \sup_{\bX\in\Theta_n:\rho(\bX,\bX_0) > jm\eps_n}\expect_\bX(1 - \phi_n)\leq \exp(-Kj^2m^2n).
  \end{align*}
  The first type I error probability bound appearing in the last display immediately implies
  \begin{align}\label{eqn:contraction_component1}
  \expect_0\left[\Pi\left\{\bX\in\Theta_n:\rho(\bX,\bX_0) \geq m\eps_n\right\} \phi_n\right]
  &\leq 
  \frac{\exp\{- (Km^2 - L)n\}}{1 - \exp(-Km^2n)}
  \end{align}
  for any $J \geq 1$. 
  \item \textbf{Component 2. }
  Invoking the aforementioned type II error probability bound in the derivation of the first component, Fubini's theorem, and the second assertion of Lemma \ref{lemma:entropy_prior_mass} leads to
  \begin{align}
  &\expect_0\left\{N_n(\{\bX\in\Theta_n:mj\eps_n < \rho(\bX,\bX_0) \leq m(j + 1)\eps_n\}) (1 - \phi_n)\right\}\nonumber\\
  &\quad = \int_{\{\bX\in\Theta_n:mj\eps_n < \rho(\bX,\bX_0) \leq m(j + 1)\eps_n\}}\expect_0\left\{(1 - \phi_n)\frac{p(\bY\mid\bX)}{p(\bY\mid\bX_0)}\right\}\Pi(\mathrm{d}\bX)\nonumber\\
  &\quad\leq \Pi\left(\bX\in\Theta_n:\rho(\bX,\bX_0)\leq m(j + 1)\eps_n\right)\sup_{\bX\in\Theta_n:\rho(\bX,\bX_0) > mj\eps_n}\expect_\bX(1 - \phi_n)\nonumber\\
  &\quad\leq \exp\left\{ - \left(\frac{K}{2}j^2m^2 - d\log m - C\right)n\right\}\left({\eps_n}\right)^{nd}.\nonumber
  \end{align}
  for some constant $C > 0$. Letting
  \[
  \Xi_n = 
  \left\{\bY:
  D_n>\exp\left\{-\left(24 + c_\pi + d\log\delta\right) n\right\}\left({\eps_n}\right)^{nd}
  \right\},
  \]
  we further obtain for some constant $\widetilde C > 0$ that
  \begin{align}
  &\expect_0\left[
  \Pi\left\{\bX\in\Theta_n:\rho(\bX,\bX_0)\in(mj\eps_n, m(j + 1)\eps_n]\right\} (1 - \phi_n)\mathbbm{1}(\bY\in\Xi_n)
  \right]\nonumber\\
  &\quad\leq \max_{\bY\in\Xi_n}\left(\frac{1}{D_n}\right)\expect_0\left\{N_n(\{\bX\in\Theta_n:mj\eps_n < \rho(\bX,\bX_0) \leq m(j + 1)\eps_n\}) (1 - \phi_n)\right\}\nonumber\\
  \label{eqn:contraction_component2}
  &\quad\leq \exp\left\{ - \left(\frac{K}{2}j^2m^2 - d\log m - \widetilde C\right)n\right\},
  \end{align}
  and Lemma \ref{lemma:evidence_LB} allows us to control the probability of $\Xi_n^c$ with $\gamma = 8$ and $\beta = \delta$:
  \begin{align}
  \label{eqn:contraction_component3}
  \prob_0\left(\Xi_n^c\right)\leq \exp\left(-\frac{nd}{2}\right).
  \end{align}
\end{itemize}
We now decompose $\expect_0[\Pi\{\bX\in\Theta_n:\rho(\bX,\bX_0) > m\eps_n\mid\bY\}]$ as follows:
\begin{align*}
\expect_0[\Pi\{\bX\in\Theta_n:\rho(\bX,\bX_0) > m\eps_n\mid\bY\}]& 
\leq \expect_0[\Pi\{\bX\in\Theta_n:\rho(\bX,\bX_0) > m\eps_n\mid\bY\}(1 - \phi_n)\mathbbm{1}(\bY\in\Xi_n)]\\
&\quad + \prob_0(\Xi_n^c) + \expect_0[\Pi\{\bX\in\Theta_n:\rho(\bX,\bX_0) > m\eps_n\mid\bY\}\phi_n].
\end{align*}
Observe that by taking a sufficiently large $m$ such that $Km^2/2 - d\log m - \widetilde C > KJm^2/4$, we have, by inequality \eqref{eqn:contraction_component2},
\begin{align*}
&\expect_0[\Pi\{\bX\in\Theta_n:\rho(\bX,\bX_0) > m\eps_n\mid\bY\}(1 - \phi_n)\mathbbm{1}(\bY\in\Xi_n)]\\
&\quad \leq \sum_{j = 1}^\infty\expect_0[\Pi\{\bX\in\Theta_n:\rho(\bX,\bX_0) \in (mj\eps_n, m(j + 1)\eps_n]\mid\bY\}(1 - \phi_n)\mathbbm{1}(\bY\in\Xi_n)]\\
&\quad \leq \exp\left\{(d\log m + \widetilde C)\right\}\sum_{j = 1}^\infty \exp\left(-\frac{K}{2}j^2m^2n\right)
\leq \exp\left\{(d\log m + \widetilde C)n\right\}\sum_{j = 1}^\infty \exp\left(-\frac{K}{2}jm^2n\right)\\
&\quad\leq \frac{\exp\left\{ - (Km^2/2 - d\log m - \widetilde C)n\right\}}{1 - \exp( - Km^2n/2)}
\leq \frac{\exp( - Km^2n/4)}{1 - \exp(-Km^2n/2)}.
\end{align*}
It follows from inequalities \eqref{eqn:contraction_component1} and \eqref{eqn:contraction_component3} that
\begin{align*}
\expect_0[\Pi\{\bX\in\Theta_n:\rho(\bX,\bX_0) > m\eps_n\mid\bY\}]
\leq \frac{2\exp( - Km^2n/4)}{1 - \exp(-Km^2n/2)} + \exp\left(-\frac{nd}{2}\right)
\end{align*}
by further requiring $Km^2 - L \geq Km^2/4$. Hence, we invoke Proposition \ref{prop:near_rootn_contraction_BRDPG} to draw the following conclusion: there exists some large constants $M_1, M$ and an absolute constant $K > 0$, such that for sufficiently large $n$,
\begin{align*}
&\expect_0\left[\Pi\left\{\rho(\bX,\bX_0) > M_1\eps_n\mid\bY\right\}\right]\\
&\quad\leq \expect_0\left[\Pi\left\{\bX\in\Theta_n:\rho(\bX,\bX_0) > M_1\eps_n\mid\bY\right\}\right]
+ \expect_0\left\{\Pi(\Theta_n\mid\bY)\right\}\\
&\quad\leq 4\exp\left(-\frac{KM_1^2n}{4}\right) + \exp\left(-\frac{nd}{2}\right) + 3\exp\left(-\frac{1}{2}nd\log n\right)
\leq 8\exp\left(-\frac{1}{2}nd\right).
\end{align*}
Namely, there exists some constant $C_0 > 0$ that is independent of $n$, such that
\begin{align*}
\expect_0\left[\Pi\left\{\rho(\bX,\bX_0) > M_1\eps_n\mid\bY\right\}\right]&\leq
8\exp\left(-\frac{1}{2}nd\right)
\end{align*}
for sufficiently large $M_1 > 0$.
The proof of the first assertion is thus completed. The second assertion directly follows from the following observation: We see from the proof of Proposition \ref{prop:near_rootn_contraction_BRDPG} (see inequality \eqref{eqn:rho_eta_equivalence}) that for any $\bX\in\Theta_n$,
\[
\frac{1}{\sqrt{n}}\inf_{\bW\in\mathbb{O}(d)}\|\bX - \bX_0\bW(\bX,\bX_0)\|_{\mathrm{F}}\leq \frac{1}{\sqrt{n}}\|\bX - \bX_0\bW(\bX,\bX_0)\|_{\mathrm{F}}\lesssim 
\rho(\bX, \bX_0). 
\]
\end{proof}

\begin{proof}[of Theorem \ref{thm:point_estimator}]
Before proving the two assertions of the theorem, we first show that $\widetilde\bP$ is close to $\bP_0 = \bX_0\bX_0\transpose$ in mean-squared error. 
Take the pseudo-metric $\rho(\bX_1,\bX_2) = (1/n)\|\bX_1\bX_1\transpose - \bX_2\bX_2\transpose\|_{\mathrm{F}}$. 
Let $M_1$ and $M_2$ be the constants provided by Theorem \ref{thm:root_n_contraction_BRDPG}. By the Jensen's inequality, we have, by Theorem \ref{thm:root_n_contraction_BRDPG}, that
\begin{align*}
\expect_0\left(\frac{1}{n^2}\|\widetilde\bP - \bX_0\bX_0\transpose\|_{\mathrm{F}}^2\right)
&\leq 
\expect_0\left\{\frac{1}{n^2}\int_{\calX^n}\|\bX\bX\transpose - \bX_0\bX_0\|_{\mathrm{F}}^2\Pi(\mathrm{d}\bX\mid\bY)
\right\}\\
& \leq \expect_0\left\{\frac{1}{n^2}\int_{\{\rho(\bX,\bX_0)\leq M_1/\sqrt{n}\}}\|\bX\bX\transpose - \bX_0\bX_0\|_{\mathrm{F}}^2\Pi(\mathrm{d}\bX\mid\bY)
\right\}\\
&\quad + \expect_0\left[\Pi\left\{\rho(\bX,\bX_0) > \frac{M_1}{\sqrt{n}}\mathrel{\bigg|}\bY\right\}\right]\left(\frac{4}{n^2}\max_{\bX\in\calX^n}\|\bX\|_{\mathrm{F}}^4\right)\\
&\leq \frac{M_1^2}{n} + 32\exp\left\{-\frac{nd}{2}\right\}\leq \frac{2M_1^2}{n}
\end{align*}

For the first assertion, we adopt the technique applied in the proof of Theorem \ref{thm:root_n_contraction_BRDPG}. Let $\bX_0$ yield singular value decompositions $\bX_0 = \bU_0\bS_0^{1/2}\bV_0\transpose$, where $\bU_0\in\mathbb{O}(n, d)$ and $\bV_0\in\mathbb{O}(d)$. Further let $\bU_0\transpose\widehat\bU = \bW_1\bSigma\bW_2\transpose$ be the singular value decomposition of $\bU_0\transpose\widehat\bU$, and let $\bW_\bU = \bW_1\bW_2\transpose$. 
Denote $\bP_0 = \bX_0\bX_0\transpose$.
Then by Lemma \ref{lemma:Decomposition} and the fact that $\|\bA\bB\|_{\mathrm{F}} \leq \|\bA\|_{\mathrm{F}}\|\bB\|_2$, over the event 
\[
\Gamma_n = \bigcap_{k = 1}^d\left\{|\sigma_k^2(\widehat\bX) - \sigma_k^2(\bX_0)| \leq \frac{n\lambda_k(\bDelta)}{4}\right\},
\]
we have,
\begin{align*}
\|\widehat\bX - \bX_0\bV_0\bW_\bU\|_{\mathrm{F}}
&\leq \|\bS_0^{-1/2}\|_2\|\widetilde\bP - \bP_0\|_{\mathrm{F}} + 
\|\widetilde\bP - \bP_0\|_{\mathrm{F}}(\|\widehat\bS^{-1/2}\|_2 + \|\bS_0^{-1/2}\|)
\\&\quad
 + \|\widetilde\bP - \bP_0\|_{\mathrm{F}}\|\widehat\bS^{-1/2}\|_2
 + 2\|\widetilde\bP - \bP_0\|_{\mathrm{F}}\|\widehat\bS^{-1/2}\|_2 + \|\sin\Theta(\widehat\bU,\bU_0)\|_2^2\|\widehat\bS^{1/2}\|_{\mathrm{F}}
\\&\quad
  + \|\bW_\bU\widehat\bS^{1/2} - \bS^{1/2}_0\bW_\bU\|_{\mathrm{F}}\\
&\leq \frac{12}{\sqrt{n\lambda_d(\bDelta)}}\|\widetilde\bP - \bP_0\|_{\mathrm{F}} + \sqrt{\frac{n\lambda_d(\bDelta)}{4}}\|\sin\Theta(\widehat\bU,\bU_0)\|_2 + \|\bW_\bU\widehat\bS^{1/2} - \bS_0^{1/2}\bW_\bU\|_{\mathrm{F}}.
\end{align*}
By Davis-Kahan theorem, 
$\|\sin\Theta(\widehat\bU,\bU_0)\|_2
\lesssim{\|\widetilde\bP - \bP_0\|_{\mathrm{F}}}/{n\lambda_d(\bDelta)}$. 
The last term $\|\bW_\bU\widehat\bS^{1/2} - \bS_0^{1/2}\bW_\bU\|_{\mathrm{F}}$ can be upper bounded by $2\|\bW_\bU\widehat\bS - \bS_0 \bW_\bU\|_{\mathrm{F}}/\sqrt{n\lambda_d(\bDelta)}$, and $\bW_\bU\widehat\bS - \bS_0 \bW_\bU$ can be decomposed as follows:
\begin{align*}
\bW_\bU\widehat\bS - \bS_0 \bW_\bU & = (\bW_\bU - \bU_0\transpose\widehat\bU)\widehat\bS +\bU_0\transpose(\widetilde\bP -\bP_0)\widehat\bU + \bS_0(\bU_0\transpose\widehat\bU - \bW_\bU).
\end{align*}
Since by Davis-Kahan theorem,
\[
\|\bW_\bU - \bU_0\transpose\widehat\bU\|_2\leq \|\sin\Theta(\widehat\bU,\bU_0)\|_2^2\lesssim \frac{\|\widetilde\bP - \bP_0\|_{\mathrm{F}}}{n},\]
it follows that
$\|\bW_\bU\widehat\bS - \bS_0 \bW_\bU\|_{\mathrm{F}}
\leq \|\bW_\bU - \bU_0\transpose\widehat\bU\|_2(\|\widehat\bS\|_{\mathrm{F}} + \|\bS_0\|_{\mathrm{F}}) + \|\widetilde\bP - \bP_0\|_{\mathrm{F}}\lesssim \|\widetilde\bP - \bP_0\|_{\mathrm{F}}$,
and that
\begin{align*}
\|\bW_\bU^{1/2}\widehat\bS^{1/2} - \bS_0^{1/2} \bW_\bU\|_{\mathrm{F}}\lesssim \frac{1}{\sqrt{n}}\|\widehat\bP - \bP_0\|_{\mathrm{F}}.
\end{align*}
Combining the above results, over the event $\Gamma_n$, we obtain $\|\widehat\bX - \bX_0\bV_0\bW_\bU\|_{\mathrm{F}}\leq (C_\bDelta/{\sqrt{n}})\|\widetilde\bP - \bP_0\|_{\mathrm{F}}$ for some constant $C_\bDelta$ depend on $\bDelta$. 
Also note that the probability of $\Gamma_n^c$ can also be bounded by Markov's inequality,
\begin{align*}
\prob_0(\Gamma_n^c) = \prob_0\left[\bigcup_{k = 1}^d\left\{|\sigma_k^2(\widehat\bX) - \sigma_k^2(\bX_0)| \geq \frac{n\lambda_k(\bDelta)}{4}\right\}\right]
 \leq \prob_0\left\{\frac{1}{n}\|\widetilde\bP - \bP_0\|_{\mathrm{F}} > \frac{\lambda_d(\bDelta)}{4}\right\}\lesssim  \frac{1}{n}.
\end{align*}
Observe that
\begin{align*}
\|\widehat\bX\|_{\mathrm{F}}^2\leq \|\widetilde\bP\|_{\mathrm{F}} \leq \int_{\calX^n}\|\bX\bX\transpose\|_{\mathrm{F}}\Pi(\mathrm{d}\bX\mid\bY)\leq \int_{\calX^n}\|\bX\|_{\mathrm{F}}^2\Pi(\mathrm{d}\bX\mid\bY)\leq n.
\end{align*}
Therefore, 
\begin{align*}
\expect_0\left(\frac{1}{n}\inf_{\bW\in\mathbb{O}(d)}\|\widehat\bX - \bX_0\bW\|_{\mathrm{F}}^2\right)
&\leq \expect_0\left\{\mathbbm{1}(\bY\in\Gamma_n)\frac{1}{n}\|\widehat\bX - \bX_0\bV_0\bW_\bU\|_{\mathrm{F}}^2\right\} + 2\prob_0(\Gamma_n^c)\\
&\leq \expect_0\left\{\frac{C_\bDelta^2}{n}\|\widetilde\bP - \bP_0\|_{\mathrm{F}}^2\right\} +
 2\prob_0\left(\Gamma_n^c\right)\lesssim\frac{1}{n}.
\end{align*}

Now we focus on the second assertion. For convenience denote $\bE = \widetilde\bP - \bX_0\bX_0\transpose$. Then for any $t > 0$, the moment generating function of $\|\bE\|_{\mathrm{F}}$ satisfies
\begin{align*}
M_{\|\bE\|_{\mathrm{F}}}(t) &= \expect_0\left[\exp\left\{t\left\|\int_{\calX^n}(\bX\bX\transpose - \bX_0\bX_0\transpose)\Pi(\mathrm{d}\bX\mid\bY)\right\|_{\mathrm{F}}\right\}\right]\\
&\leq \expect_0\left[\exp\left\{t\int_{\calX^n}\left\|\bX\bX\transpose - \bX_0\bX_0\transpose\right\|_{\mathrm{F}}\Pi(\mathrm{d}\bX\mid\bY)\right\}\right]\\
&\leq \expect_0\left[\exp\left\{t\int_{\{\|\bX\bX\transpose - \bX_0\bX_0\transpose\|_{\mathrm{F}} \leq \sqrt{n}M_1 \}}\left\|\bX\bX\transpose - \bX_0\bX_0\transpose\right\|_{\mathrm{F}}\Pi(\mathrm{d}\bX\mid\bY)\right\}\right]\\
&\quad + \expect_0\left\{\Pi\left(\{\|\bX\bX\transpose - \bX_0\bX_0\transpose\|_{\mathrm{F}} > \sqrt{n}M_1\mid\bY\right)\right\}\exp\left\{t\sup_{\bX\in\calX^n}\|\bX\bX\transpose - \bX_0\bX_0\transpose\|_{\mathrm{F}}\right\}\\
&\leq \exp\left(tM_1\sqrt{n}\right) + 8\exp\left(2nt - \frac{nd}{2}\right).
\end{align*}
It follows from Chernoff bound with $t = d/4$ that, 
\begin{align*}
\prob_0\left(\|\bE\|_{\mathrm{F}} > 2M_1\sqrt{n}\right)
\leq  2\exp\left(-\frac{1}{4}M_1d\sqrt{n}\right)
\end{align*}
for sufficiently large $n$. The proof is completed by observing that by Davis-Kahan theorem,
\[
\|\widehat\bU - \bU_0\bW_\bU\|_{\mathrm{F}}\leq \sqrt{2d}\|\sin\Theta(\widehat\bU, \bU_0)\|_2\leq \frac{2\sqrt{2d}\|\bE\|_{\mathrm{F}}}{\lambda_d(\bX_0\bX_0\transpose)}\leq \frac{4\sqrt{2d}\|\bE\|_{\mathrm{F}}}{n\lambda_d(\bDelta)}.
\]
\end{proof}

\section{Proofs for Section \ref{sec:sample_application_community_detection_in_stochastic_blockmodels}} 
\label{sec:proofs_for_section_sub:sample_application_community_detection_in_stochastic_blockmodels}

\begin{proof}[of Theorem \ref{thm:PSE_clustering}]
Assume that \emph{a posteriori} the event $\{\|\bX - \bX_0\bW(\bX,\bX_0)\|_{\mathrm{F}}\leq M_2\}$ occurs, where $\bW(\bX, \bX_0) = \arginf_{\bW\in\mathbb{O}(d)}\|\bX - \bX_0\bW\|_{\mathrm{F}}$. Observe that by definition and triangle inequality, 
\begin{align*}
\|\bC(\bX) - \bX_0\bW(\bX_0,\bX)\|_{\mathrm{F}}
&\leq \|\bC(\bX) - \bX\|_{\mathrm{F}} + \|\bX - \bX_0\bW(\bX_0,\bX)\|_{\mathrm{F}}\\
&\leq 2\|\bX - \bX_0\bW(\bX_0,\bX)\|_{\mathrm{F}}\leq{2M_2}.
\end{align*}
Now we argue that the number of rows $\calV = \{i\in[n]:\|(\bC(\bX))_{i*} - \bW(\bX_0,\bX)\transpose\bx_{0i}\|> \xi/2\}$ is no greater than $16M_2^2/\xi^2$ by contradiction. Assuming otherwise, then we obtain
\begin{align*}
\|\bC(\bX) - \bX_0\bW(\bX_0,\bX)\|_{\mathrm{F}}^2 > \left(\frac{16M^2_2}{\xi^2}\right)\left(\frac{\xi}{2}\right)^2 = 4M_2^2,
\end{align*}
contradicting with the previous observation. Namely, $|\calV^c|\geq n - 16M^2/\xi^2$.  Consequently, for any $i,j\in\calV^c$, $(\bC(\bX))_{i*} = (\bC(\bX))_{j*}$, we see that
\begin{align*}
\|\bx_{0i} - \bx_{0j}\|_2 
& = \|\bW(\bX_0,\bX)\transpose(\bx_{0i} - \bx_{0j})\|_2\\
& \leq \|(\bC(\bX))_{i*} - \bW(\bX_0,\bX)\transpose\bx_{0i} \| + \|(\bC(\bX))_{j*} - \bW(\bX_0,\bX)\transpose\bx_{0j} \| \leq \xi,
\end{align*}
implying that $\bx_{0i} = \bx_{0j}$ by assumption. Note that $n_k \geq |\calV|$ for all $k$, \emph{i.e.}, $\{\bx_{0i}:i\in\calV^c\} = \{\bx_{0k}^*:k\in[K]\}$, it follows that for each $k\in[K]$, $B_{\|\cdot\|_2}(\bW(\bX_0,\bX)\transpose\bx_{0k}^*, \xi/2)$ contains at least one element of $\{(\bC(\bX))_{i*}:i\in\calV^c\}$. Since $B_{\|\cdot\|_2}(\bW(\bX_0,\bX)\transpose\bx_{0k}^*, \xi/2)$ are disjoint by assumption, and there are only $K$ distinct rows in $\bC(\bX)$, it follows directly from the pigeonhole principle that each $B_{\|\cdot\|_2}(\bW(\bX_0,\bX)\transpose\bx_{0k}^*, \xi/2)$ contains exactly one element of $\{(\bC(\bX))_{i*}:i\in\calV^c\}$. Consequently, if $\bx_{0i} = \bx_{0j} = \bx_{0k}^*$ for some $i,j\in\calV^c$ and $k\in[K]$, then $\bC(\bX)_{i*},\bC(\bX)_{j*}\in B_{\|\cdot\|_2}(\bW(\bX_0,\bX)\transpose\bx_{0k}^*, \xi/2)$, implying that $\bC(\bX)_{i*} = \bC(\bX)_{j*}$. 

The above argument can be briefly stated as follows: $\bx_{0i} = \bx_{0j}$ if and only if $\bC(\bX)_{i*} = \bC(\bX)_{j*}$. This immediately implies that
\[
\inf_{\sigma\in\calS_K}d_H(\sigma\circ\tau(\cdot;\bX_0), \tau(\cdot;\bX)) \leq \frac{16M_2^2}{\xi^2},
\]
and the first assertion is proved by an application of Theorem \ref{thm:root_n_contraction_BRDPG}. 

To prove the second assertion, we need to apply the large deviation bound in Theorem \ref{thm:point_estimator}. Note that for any $i,j\in[n]$ with $\bx_{0i}\neq \bx_{0j}$, 
\begin{align*}
\|\bx_{0i} - \bx_{0j}\|_2^2
& = (\be_i - \be_j)\transpose\bX_0\bX_0\transpose(\be_i - \be_j) = \|\bS^{1/2}_0\bU_0\transpose(\be_i - \be_j)\|_2^2\leq \|\bX_0\|_{\mathrm{F}}^2\|\bU_0\transpose(\be_i - \be_j)\|_2\\
&\leq n(\be_i - \be_j)\transpose\bU_0\bU_0\transpose(\be_i - \be_j) = n\|(\bU_0)_{i*} - (\bU_0)_{j*}\|_2^2. 
\end{align*}
By assumption, this implies that $\|(\bU_0)_{i*} - (\bU_0)_{j*}\|_2 > \xi/\sqrt{n}$. Assume the event $\left\{\|\widehat\bU - \bU_0\bW_\bU\|_{\mathrm{F}} \leq M'/\sqrt{n}\right\}
$ occurs with respect to $\prob_0$, where 
$M' = {8M_1\sqrt{2d}}/{\lambda_d(\bDelta)}$ is a constant. Similarly, 
\[
\|\bC(\widehat\bU) - \bU_0\bW_\bU\|_{\mathrm{F}}\leq 2\|\widehat\bU - \bU_0\bW_\bU\|\leq \frac{2M'}{\sqrt{n}}.
\]
An argument that is similar to that for the first assertion (up to a factor of $1/\sqrt{n}$) shows that 
\[
\inf_{\sigma\in\calS_K}d_H(\sigma\circ\tau(\cdot;\widehat\bU), \tau(\cdot;\bU_0))\leq \frac{16M'^2}{\xi^2}.
\]
Namely, 
\begin{align*}
\prob_0\left(\inf_{\sigma\in\calS_K}d_H(\tau(\cdot;\widehat\bU), \tau(\cdot;\bU_0)) > \frac{16M'^2}{\xi^2}\right)\leq  2\exp\left(-\frac{1}{4}M_1d\sqrt{n}\right).
\end{align*}
It follows immediately from Borel-Cantelli lemma that 
\begin{align*}
\prob_0\left(\inf_{\sigma\circ\sigma\in\calS_K}d_H(\sigma\circ\tau(\cdot;\widehat\bU), \tau(\cdot;\bU_0)) \leq \frac{16M'^2}{\xi^2}\text{ almost always}\right) = 1. 
\end{align*}
The proof is completed by plugging-in $M'$. 
\end{proof}

\section{Proofs for Section \ref{sec:spectral_based_bayesian_estimation}} 
\label{sec:proof_of_theorem_GSE}


\begin{lemma}\label{lemma:sub_Gaussian_concentration}
Let $\bE\in\mathbb{R}^{n\times n}$ be a symmetric random matrix with $(y_{ij}:1\leq i< j\leq n)$ being independent, and let $\expect_0(\bE) = \zero_{n\times n}$. Assume that $\bE$ are sub-Gaussian, \emph{i.e.}, there exists some constant $\tau > 0$, such that for all $\bA\in\mathbb{R}^{n\times n}$ with $\|\bA\|_{\mathrm{F}}^2 = 1$, for all $t > 0$,
$\prob_0\left(|\mathrm{Tr}\left(\bA\transpose\bY\right)| > t\right)\leq 2\mathrm{e}^{-\tau t^2}$.
For any $\bX,\bX_0\in\mathbb{R}^{n\times d}$ and for any $t > 0$, 
\[
\prob_0\left(\sup_{\bX\in\mathbb{R}^{n\times d}}\left|\left\langle\bE,\frac{\bX\bX\transpose - \bX_0\bX_0\transpose}{\|\bX\bX\transpose - \bX_0\bX_0\transpose\|_{\mathrm{F}}}\right\rangle_{\mathrm{F}}\right| > nt\right)\leq  6\exp\left(3nd - \frac{\tau n^2t^2}{4}\right) .
\]
\end{lemma}
\begin{proof}
The proof is based on a popular discretization and covering technique (see, for example, \citealp{tight_oracle_inequalities}). First observe that
\begin{align*}
\left\{\sup_{\bX\in\mathbb{R}^{n\times d}}
\left|\left\langle\bE,\frac{\bX\bX\transpose - \bX_0\bX_0\transpose}{\|\bX\bX\transpose - \bX_0\bX_0\transpose\|_{\mathrm{F}}}\right\rangle_{\mathrm{F}}\right| > nt
\right\}
&\subset \left\{\sup_{\mathrm{rank}(\bB)\leq 2d, \|\bB\|_{\mathrm{F}} = 1}
\left|\left\langle\bE,{\bB}\right\rangle_{\mathrm{F}}\right| > nt
\right\}.
\end{align*}
Let $\calD(1/6)$ be an $1/6$-net of $\{\bS = \mathrm{diag}(\sigma_1,\ldots,\sigma_{2d}):\|\bS\|_{\mathrm{F}} = 1,\sigma_i\geq 0\}$, and $\calO(1/6)$ be an $1/6$-net of $\{\bU\in\mathbb{R}^{n\times 2d}:\|\bU\|_{\mathrm{F}} = 1\}$. Clearly, $|\calO(1/6)|\leq (18)^{nd}$ and $|\calD(1/6)|\leq (18)^{2d}$ due to the covering number bounds of the Euclidean space \citep{pollard1990empirical}. For any $\bB$ with $\mathrm{rank}(\bB) \leq 2d$ and $\|\bB\|_{\mathrm{F}} = 1$, let $\bB$ admits singular value decomposition $\bB = \bU\bS\bV\transpose$, where $\bU,\bV\in\mathbb{O}(n, 2d)$. Then there exists some $\widetilde\bU,\widetilde\bV\in\calO(1/6)$ and $\widetilde\bS\in\calD(1/6)$, such that $\|\bU - \widetilde\bU\|_{\mathrm{F}} < 1/6$, $\|\bV - \widetilde\bV\|_{\mathrm{F}} < 1/6$, and $\|\bS - \widetilde\bS\|_{\mathrm{F}} < 1/6$. We proceed to derive
\begin{align*}
\left|\left\langle\bE,{\bB}\right\rangle_{\mathrm{F}}\right|& = 
\left|\left\langle\bE, \bU\bS\bV\transpose\right\rangle_{\mathrm{F}}\right|\\
&\leq \left|\left\langle\bE, \bU\bS(\bV - \widetilde\bV)\transpose\right\rangle_{\mathrm{F}}\right| + \left|\left\langle\bE, \bU(\bS - \widetilde\bS)\widetilde\bV\transpose\right\rangle_{\mathrm{F}}\right| + \left|\left\langle\bE, (\bU - \widetilde\bU)\widetilde\bS\widetilde\bV\transpose\right\rangle_{\mathrm{F}}\right|
\\&\quad
 + \left|\left\langle\bE, \widetilde\bU\widetilde\bS\widetilde\bV\transpose\right\rangle_{\mathrm{F}}\right|\\
&\leq \left|\left\langle\bE, \frac{\bU\bS(\bV - \widetilde\bV)\transpose}{\|\bU\bS(\bV - \widetilde\bV)\transpose\|_{\mathrm{F}}}\right\rangle_{\mathrm{F}}\right|\|\bV - \widetilde\bV\|_{\mathrm{F}}
+ \left|\left\langle\bE, \frac{\bU(\bS - \widetilde\bS)\widetilde\bV\transpose}{\|\bU(\bS - \widetilde\bS)\widetilde\bV\transpose\|_{\mathrm{F}}}\right\rangle_{\mathrm{F}}\right|\|\bS - \widetilde\bS\|_{\mathrm{F}}
\\&\quad
 + \left|\left\langle\bE, \frac{(\bU - \widetilde\bU)\widetilde\bS\widetilde\bV\transpose}{\|(\bU - \widetilde\bU)\widetilde\bS\widetilde\bV\transpose\|_{\mathrm{F}}}\right\rangle_{\mathrm{F}}\right|\|\bU - \widetilde\bU\|_{\mathrm{F}} + \sup_{\widetilde\bU,\widetilde\bV\in\calO(1/6),\bS\in\calD(1/6)}\left|\left\langle\bE, \widetilde\bU\widetilde\bS\widetilde\bV\transpose\right\rangle_{\mathrm{F}}\right|\\
&\leq \frac{1}{2}\sup_{\mathrm{rank}(\bB)\leq 2d,\|\bB\|_{\mathrm{F}} = 1}|\langle\bE,\bB\rangle_{\mathrm{F}}| + \sup_{\widetilde\bU,\widetilde\bV\in\calO(1/6),\bS\in\calD(1/6)}\left|\left\langle\bE, \widetilde\bU\widetilde\bS\widetilde\bV\transpose\right\rangle_{\mathrm{F}}\right|.
\end{align*}
Hence, we obtain, after taking the supremum with respect to $\{\bB:\mathrm{rank}(\bB)\leq 2d,\|\bB\|_{\mathrm{F}} = 1\}$, that
\[
\sup_{\mathrm{rank}(\bB)\leq 2d,\|\bB\|_{\mathrm{F}} = 1}|\langle\bE,\bB\rangle_{\mathrm{F}}| \leq 2 \sup_{\widetilde\bU,\widetilde\bV\in\calO(1/6),\bS\in\calD(1/6)}\left|\left\langle\bE, \widetilde\bU\widetilde\bS\widetilde\bV\transpose\right\rangle_{\mathrm{F}}\right|.
\]
Therefore, by the union bound, we obtain
\begin{align*}
&\prob_0\left(\sup_{\bX\in\mathbb{R}^{n\times d}}\left|\left\langle\bE,\frac{\bX\bX\transpose - \bX_0\bX_0\transpose}{\|\bX\bX\transpose - \bX_0\bX_0\transpose\|_{\mathrm{F}}}\right\rangle_{\mathrm{F}}\right| > nt\right)
\\&\quad
\leq \prob\left(\sup_{\widetilde\bU,\widetilde\bV\in\calO(1/6),\bS\in\calD(1/6)}\left|\left\langle\bE, \widetilde\bU\widetilde\bS\widetilde\bV\transpose\right\rangle_{\mathrm{F}}\right| > \frac{nt}{2}\right)
\\&\quad
\leq\sum_{\widetilde\bU\in\calO(1/6)}\sum_{\widetilde\bV\in\calO(1/6)}\sum_{\widetilde\bS\in\calD(1/6)}\prob\left(\left|\left\langle\bE, \widetilde\bU\widetilde\bS\widetilde\bV\transpose\right\rangle_{\mathrm{F}}\right| > \frac{nt}{2}\right)
\\&\quad
\leq 6(18)^{nd}\exp(-\tau nt^2/4)\leq 6\exp\left(3nd - \frac{\tau n^2t^2}{4}\right),
\end{align*}
where we have invoke the condition that $\bE$ has sub-Gaussian entries. The proof is thus completed.
\end{proof}

\begin{proof}[of Theorem \ref{thm:pseudo_posterior_contraction}]
Let $\eps_n = \sqrt{d(\log n)/n}$. Simple algebra shows that 
\[
-\frac{1}{2}\|\bY - \bX\bX\transpose\|_{\mathrm{F}}^2 + \frac{1}{2}\|\bY - \bX_0\bX_0\transpose\|_{\mathrm{F}}^2
 = -\frac{1}{2}\|\bX\bX\transpose - \bX_0\bX_0\transpose\|_{\mathrm{F}}^2 + \langle\bE, \bX\bX\transpose - \bX_0\bX_0\transpose\rangle_{\mathrm{F}},
\]
where $\bE = \bY - \expect_0\bY$. For any $\alpha > 0$, denote
\[
\calE_n(\alpha) = \left\{\sup_{\bX\in\mathbb{R}^{n\times d}}\left|\left\langle\bE,\frac{\bX\bX\transpose - \bX_0\bX_0\transpose}{\|\bX\bX\transpose - \bX_0\bX_0\transpose\|_{\mathrm{F}}}\right\rangle_{\mathrm{F}}\right| \leq \alpha n\eps_n\right\}.
\]
Since $\bE$ is sub-Gaussian, we can invoke Lemma \ref{lemma:sub_Gaussian_concentration} and obtain
\begin{align*}
\prob_0\left\{\calE_n(\alpha)^c\right\}  \leq 6\exp\left(3nd - \frac{\tau n^2\alpha^2\eps_n^2}{4}\right) = 6\exp\left\{-\left(\frac{\alpha^2\tau}{4}\log n - 3\right)dn\right\}.
\end{align*}
Denote $\calU_n(\gamma) = \{\|\bX\bX\transpose - \bX_0\bX_0\transpose\|_{\mathrm{F}} \leq n\gamma\}$. Then over the event $\calE_n(\alpha)$, the denominator $D_n^G$ can be lower bounded as follows:
\begin{align*}
D_n^G& \geq \int_{\calU_n(\eps_n)}\exp\left\{ - \left|\left\langle\bE,\frac{\bX\bX\transpose - \bX_0\bX_0\transpose}{\|\bX\bX\transpose - \bX_0\bX_0\transpose\|_{\mathrm{F}}}\right\rangle_{\mathrm{F}}\right|\|\bX\bX\transpose - \bX_0\bX_0\transpose\|_{\mathrm{F}}
\right.
\\&\quad\quad\quad\quad\quad\quad\quad
\left. -\frac{1}{2}\|\bX\bX\transpose - \bX_0\bX_0\transpose\|_{\mathrm{F}}^2
\right\}\Pi(\mathrm{d}\bX)\\
& \geq \int_{\calU_n(\eps_n)}\exp\left\{-\frac{1}{2}\|\bX\bX\transpose - \bX_0\bX_0\transpose\|_{\mathrm{F}}^2 - \alpha n\eps_n\|\bX\bX\transpose - \bX_0\bX_0\transpose\|_{\mathrm{F}}\right\}\Pi(\mathrm{d}\bX)\\
& \geq \int_{\calU_n(\eps_n)}\exp\left\{-\frac{1}{2}\|\bX\bX\transpose - \bX_0\bX_0\transpose\|_{\mathrm{F}}^2 - \frac{1}{2}\alpha^2 n^2\eps_n^2 - \frac{1}{2}\|\bX\bX\transpose - \bX_0\bX_0\transpose\|_{\mathrm{F}}^2\right\}\Pi(\mathrm{d}\bX)\\
& \geq\Pi\{\calU_n(\eps_n)\}\exp\left\{-\left(1 + \frac{\alpha^2}{2}\right)n^2\eps_n^2\right\}.
\end{align*}
Observe that
\begin{align*}
\|\bX\bX\transpose - \bX_0\bX_0\transpose\|_{\mathrm{F}}
&\leq \|(\bX - \bX_0)(\bX - \bX_0)\transpose\|_{\mathrm{F}} + \|\bX_0(\bX - \bX_0)\transpose\|_{\mathrm{F}} + \|(\bX - \bX_0)\bX_0\transpose\|_{\mathrm{F}}\\
&\leq \|\bX - \bX_0\|_{\mathrm{F}}^2 + 2\sqrt{n}\|\bX - \bX_0\|_{\mathrm{F}}= (2\sqrt{n} + \|\bX - \bX_0\|_{\mathrm{F}})\|\bX - \bX_0\|_{\mathrm{F}}.
\end{align*}
It follows that
\begin{align*}
\left\{\bX:\|\bX - \bX_0\|_{\mathrm{F}}\leq\frac{\sqrt{n}\eps_n}{3}\right\}\subset\{\|\bX\bX\transpose - \bX_0\bX_0\transpose\|_{\mathrm{F}}\leq n\eps_n\} = \calU_n(\eps_n).
\end{align*}
Note that the concentration of Gaussian distribution can be lower bounded by the Anderson's lemma:
\begin{align*}
\Pi\{\calU_n(\eps_n)\}& \geq \Pi\left(\bX:\|\bX - \bX_0\|_{\mathrm{F}}\leq\frac{\sqrt{n}\eps_n}{3}\right)
\geq\exp\left(-\frac{1}{2\sigma^2}\|\bX_0\|_{\mathrm{F}}^2\right)\prod_{i = 1}^n\prod_{k = 1}^d\Pi\left(x_{jk}^2\leq\frac{\eps_n^2}{9d}\right)\\
&= \exp\left(-\frac{1}{2\sigma^2}\|\bX_0\|_{\mathrm{F}}^2\right)\prod_{i = 1}^n\prod_{k = 1}^d\left\{2\Phi\left(\frac{\eps_n}{3\sigma\sqrt{d}}\right) - 1\right\}\\
&\geq
\exp\left\{-\left(\frac{1}{2\sigma^2} + d\right)n - nd\left|\log\frac{\eps_n}{3\sigma\sqrt{d}}\right|\right\}\\
&\geq\exp\left\{-\left(\frac{1}{2\sigma^2} + d + d|\log3\sigma|\right)n - \frac{1}{2}nd\log n\right\}.
\end{align*}
Hence, over the event $\calE_n(\alpha)$, we obtain 
\begin{align*}
D_n^G&\geq \Pi\{\calU_n(\eps_n)\}\exp\left\{-\left(1 + \frac{\alpha^2}{2}\right)n^2\eps_n^2\right\}\\
&\geq\exp\left\{-\left(\frac{\alpha^2 + 3}{2}\right)nd\log n -\left(\frac{1}{2\sigma^2} + d + d|\log3\sigma|\right)n\right\}.
\end{align*}
We proceed to bound $\expect_0[\Pi\{\calU_n^c(M\eps_n)\mid\bY\}]$ as follows:
\begin{align*}
&\expect_0[\Pi\{\calU_n^c(M\eps_n)\mid\bY\}]\\
&\quad\leq \expect_0\left[\mathbbm{1}\{\bY\in\calE_n(\alpha)\}\left\{\frac{N_n^G(\calU_n^c(M\eps_n))}{D_n^G}\right\}\right] + \prob_0\{\calE_n^c(\alpha)\}\\
&\quad\leq \exp\left\{\left(\frac{\alpha^2 + 3}{2}\right)nd\log n + \left(\frac{1}{2\sigma^2} + d + d|\log3\sigma|\right)n\right\}\\
&\quad\quad\times \int_{\calU_n^c(M\eps_n)}\expect_0\left[\mathbbm{1}\{\bY\in\calE_n(\alpha)\}\exp\left\{-\frac{1}{2}\|\bX\bX\transpose - \bX_0\bX_0\transpose\|_{\mathrm{F}}^2 + \langle\bE,\bX\bX\transpose - \bX_0\bX_0\transpose\rangle\right\}\right]\Pi(\mathrm{d}\bX)\\
&\quad\quad + 6\exp\left\{ - \left(\frac{\alpha^2\tau}{4}\log n - 3\right)nd\right\}\\
&\quad\leq \exp\left\{\left(\frac{\alpha^2 + 3}{2}\right)nd\log n + \left(\frac{1}{2\sigma^2} + d + d|\log3\sigma|\right)n\right\}\\
&\quad\quad\times \int_{\calU_n^c(M\eps_n)}\exp\left\{-\frac{1}{2}\|\bX\bX\transpose - \bX_0\bX_0\transpose\|_{\mathrm{F}}^2 + \alpha n\eps_n\|\bX\bX\transpose - \bX_0\bX_0\transpose\|_{\mathrm{F}}\right\}\Pi(\mathrm{d}\bX)\\
&\quad\quad + 6\exp\left\{ - \left(\frac{\alpha^2\tau}{4}\log n - 3\right)dn\right\}\\
&\quad\leq \exp\left\{\left(\frac{\alpha^2 + 3}{2}\right)n\log n + \left(\frac{1}{2\sigma^2} + d + d|\log3\sigma|\right)n\right\}\\
&\quad\quad\times \int_{\calU_n^c(M\eps_n)}\exp\left\{-\frac{1}{2}\|\bX\bX\transpose - \bX_0\bX_0\transpose\|_{\mathrm{F}}^2 + 2\alpha^2 n^2\eps_n^2 + \frac{1}{8}\|\bX\bX\transpose - \bX_0\bX_0\transpose\|_{\mathrm{F}}^2\right\}\Pi(\mathrm{d}\bX)\\
&\quad\quad + 6\exp\left\{\left(\frac{\alpha^2\tau}{4}\log n - 3\right)nd\right\}\\
&\quad\leq \exp\left\{\left(\frac{1}{2\sigma^2} + d + d|\log3\sigma|\right)n-\left(\frac{3}{8}M^2 - \frac{5\alpha^2 + 3}{2}\right) nd\log n\right\}\\
&\quad\quad + 6\exp\left\{ - \left(\frac{\alpha^2\rho^2}{4}\log n - 3\right)nd\right\},
\end{align*}
where the second inequality is due to Fubini's theorem, and the fourth inequality is due to the fact that $ab\leq 2a^2 + b^2/8$ for any $a,b>0$. Hence, taking $\alpha = \sqrt{(M^2 + 6)/10}$, then for sufficiently large $M$, we see that
\begin{align*}
\expect_0[\Pi\{\calU_n^c(M\eps_n)\mid\bY\}]&\leq \exp\left(-\frac{1}{16}M^2nd\log n\right) + 6\exp\left\{-\frac{(M^2 + 6)\rho^2}{80}nd\log n\right\}\\
&\leq7\exp\left\{-\min\left(\frac{1}{16},\frac{\tau}{80}\right)M^2nd\log n\right\}.
\end{align*}
Now set $\Xi_n = \{\bX\in\mathbb{R}^{n\times d}:(1/n)\|\bX\bX\transpose - \bX_0\bX_0\transpose\|_{\mathrm{F}}\leq \sqrt{(Md\log n)/n}\}$. Following the same argument for deriving inequality \eqref{eqn:rho_eta_equivalence}, we see that there exists some constant $\widetilde C> 0$, such that for any $\bX\in\Xi_n$ and sufficiently large $n$, 
\begin{align*}
\inf_{\bW\in\mathbb{O}(d)}\|\bX - \bX_0\bW\|_{\mathrm{F}}\lesssim
 \sqrt{\frac{d}{n}}\|\bX\bX\transpose - \bX_0\bX_0\transpose\|_{\mathrm{F}}.
\end{align*}
Therefore, there exists some constant $M' > 0$, such that for any $\bX\in\Xi_n$,
\[
\frac{1}{M'n}\inf_{\bW\in\mathbb{O}(d)}\|\bX - \bX_0\bW\|_{\mathrm{F}}^2 \leq \frac{1}{Mn^2}\|\bX\bX\transpose - \bX_0\bX_0\transpose\|_{\mathrm{F}}^2,
\]
and hence,
\begin{align*}
&\expect_0\left\{
\Pi_G\left(\frac{1}{n}\inf_{\bW\in\mathbb{O}(d)}\|\bX - \bX_0\bW\|_{\mathrm{F}}^2 > \frac{M'd\log n}{n}\mathrel{\Big|}\bY\right)
\right\}\\
&\quad\leq \expect_0\left\{
\Pi_G\left(\bX\in\Xi_n:\frac{1}{n}\inf_{\bW\in\mathbb{O}(d)}\|\bX - \bX_0\bW\|_{\mathrm{F}}^2 > \frac{M'd\log n}{n}\mathrel{\Big|}\bY\right)
\right\} + \expect_0\left\{
\Pi_G\left(\bX\in\Xi_n^c\mid \bY\right)
\right\}\\
&\quad\leq \expect_0\left\{
\Pi_G\left(\frac{1}{n^2}\|\bX\bX\transpose - \bX_0\bX_0\transpose\|_{\mathrm{F}}^2 > \frac{Md\log n}{n}\mathrel{\Big|}\bY\right)
\right\} + \expect_0\left\{
\Pi_G\left(\bX\in\Xi_n^c\mid \bY\right)
\right\}\\
&\quad\leq 2\expect_0\left\{
\Pi_G\left(\bX\in\Xi_n^c\mid \bY\right)
\right\}\leq 14\exp\left\{-\min\left(\frac{1}{16},\frac{\tau}{80}\right)M^2nd\log n\right\},
\end{align*}
completing the proof. 
\end{proof}

\bibliographystyle{apalike}
\bibliography{reference1,reference2}

\begin{thebibliography}{}

\bibitem[Athreya et~al., 2018a]{JMLR:v18:17-448}
Athreya, A., Fishkind, D.~E., Tang, M., Priebe, C.~E., Park, Y., Vogelstein,
  J.~T., Levin, K., Lyzinski, V., Qin, Y., and Sussman, D.~L. (2018a).
\newblock Statistical inference on random dot product graphs: a survey.
\newblock {\em Journal of Machine Learning Research}, 18(226):1--92.

\bibitem[Athreya et~al., 2016]{athreya2016limit}
Athreya, A., Priebe, C.~E., Tang, M., Lyzinski, V., Marchette, D.~J., and
  Sussman, D.~L. (2016).
\newblock A limit theorem for scaled eigenvectors of random dot product graphs.
\newblock {\em Sankhya A}, 78(1):1--18.

\bibitem[Athreya et~al., 2018b]{athreya2018estimation}
Athreya, A., Tang, M., Park, Y., and Priebe, C.~E. (2018b).
\newblock On estimation and inference in latent structure random graphs.
\newblock {\em arXiv preprint arXiv:1806.01401}.

\bibitem[Bickel and Chen, 2009]{Bickel21068}
Bickel, P.~J. and Chen, A. (2009).
\newblock A nonparametric view of network models and newman{\textendash}girvan
  and other modularities.
\newblock {\em Proceedings of the National Academy of Sciences},
  106(50):21068--21073.

\bibitem[Bickel et~al., 2011]{bickel2011}
Bickel, P.~J., Chen, A., and Levina, E. (2011).
\newblock The method of moments and degree distributions for network models.
\newblock {\em Ann. Statist.}, 39(5):2280--2301.

\bibitem[Cai et~al., 2013]{cai2013sparse}
Cai, T.~T., Ma, Z., and Wu, Y. (2013).
\newblock Sparse {PCA}: Optimal rates and adaptive estimation.
\newblock {\em The Annals of Statistics}, 41(6):3074--3110.

\bibitem[Candes and Plan, 2011]{tight_oracle_inequalities}
Candes, E.~J. and Plan, Y. (2011).
\newblock Tight oracle inequalities for low-rank matrix recovery from a minimal
  number of noisy random measurements.
\newblock {\em IEEE Transactions on Information Theory}, 57(4):2342--2359.

\bibitem[Chatterjee, 2015]{chatterjee2015}
Chatterjee, S. (2015).
\newblock Matrix estimation by universal singular value thresholding.
\newblock {\em Ann. Statist.}, 43(1):177--214.

\bibitem[Choi et~al., 2012]{10.1093/biomet/asr053}
Choi, D.~S., Wolfe, P.~J., and Airoldi, E.~M. (2012).
\newblock {Stochastic blockmodels with a growing number of classes}.
\newblock {\em Biometrika}, 99(2):273--284.

\bibitem[Fortunato, 2010]{FORTUNATO201075}
Fortunato, S. (2010).
\newblock Community detection in graphs.
\newblock {\em Physics Reports}, 486(3):75 -- 174.

\bibitem[Fraley et~al., 2012]{fraley2012mclust}
Fraley, C., Raftery, A.~E., Murphy, T.~B., and Scrucca, L. (2012).
\newblock mclust version 4 for r: normal mixture modeling for model-based
  clustering, classification, and density estimation.
\newblock Technical report, Technical report.

\bibitem[Ghosal et~al., 2000]{ghosal2000convergence}
Ghosal, S., Ghosh, J.~K., and Van Der~Vaart, A.~W. (2000).
\newblock Convergence rates of posterior distributions.
\newblock {\em Annals of Statistics}, 28(2):500--531.

\bibitem[Ghosal et~al., 2007]{ghosal2007convergence}
Ghosal, S., Van Der~Vaart, A., et~al. (2007).
\newblock Convergence rates of posterior distributions for noniid observations.
\newblock {\em The Annals of Statistics}, 35(1):192--223.

\bibitem[Girvan and Newman, 2002]{Girvan7821}
Girvan, M. and Newman, M. E.~J. (2002).
\newblock Community structure in social and biological networks.
\newblock {\em Proceedings of the National Academy of Sciences},
  99(12):7821--7826.

\bibitem[Goldenberg et~al., 2010]{MAL-005}
Goldenberg, A., Zheng, A.~X., Fienberg, S.~E., and Airoldi, E.~M. (2010).
\newblock A survey of statistical network models.
\newblock {\em Foundations and TrendsÂ® in Machine Learning}, 2(2):129--233.

\bibitem[Hoff et~al., 2002]{doi:10.1198/016214502388618906}
Hoff, P.~D., Raftery, A.~E., and Handcock, M.~S. (2002).
\newblock Latent space approaches to social network analysis.
\newblock {\em Journal of the American Statistical Association},
  97(460):1090--1098.

\bibitem[Levin et~al., 2017]{levin2017central}
Levin, K., Athreya, A., Tang, M., Lyzinski, V., and Priebe, C.~E. (2017).
\newblock A central limit theorem for an omnibus embedding of random dot
  product graphs.
\newblock {\em arXiv preprint arXiv:1705.09355}.

\bibitem[Lloyd, 1982]{lloyd1982least}
Lloyd, S. (1982).
\newblock Least squares quantization in pcm.
\newblock {\em IEEE transactions on information theory}, 28(2):129--137.

\bibitem[Lov{\'a}sz, 2012]{lovasz2012large}
Lov{\'a}sz, L. (2012).
\newblock {\em Large networks and graph limits}, volume~60.
\newblock American Mathematical Soc.

\bibitem[Lyzinski et~al., 2017a]{lyzinski2017consistent}
Lyzinski, V., Levin, K., and Priebe, C.~E. (2017a).
\newblock On consistent vertex nomination schemes.
\newblock {\em arXiv preprint arXiv:1711.05610}.

\bibitem[Lyzinski et~al., 2014]{lyzinski2014}
Lyzinski, V., Sussman, D.~L., Tang, M., Athreya, A., and Priebe, C.~E. (2014).
\newblock Perfect clustering for stochastic blockmodel graphs via adjacency
  spectral embedding.
\newblock {\em Electron. J. Statist.}, 8(2):2905--2922.

\bibitem[Lyzinski et~al., 2017b]{7769223}
Lyzinski, V., Tang, M., Athreya, A., Park, Y., and Priebe, C.~E. (2017b).
\newblock Community detection and classification in hierarchical stochastic
  blockmodels.
\newblock {\em IEEE Transactions on Network Science and Engineering},
  4(1):13--26.

\bibitem[Pati and Bhattacharya, 2015]{pati2015optimal}
Pati, D. and Bhattacharya, A. (2015).
\newblock Optimal {B}ayesian estimation in stochastic block models.
\newblock {\em arXiv preprint arXiv:1505.06794}.

\bibitem[Pollard, 1990]{pollard1990empirical}
Pollard, D. (1990).
\newblock Empirical processes: theory and applications.
\newblock In {\em NSF-CBMS regional conference series in probability and
  statistics}, pages i--86. JSTOR.

\bibitem[Priebe et~al., 2017]{priebe2017semiparametric}
Priebe, C.~E., Park, Y., Tang, M., Athreya, A., Lyzinski, V., Vogelstein,
  J.~T., Qin, Y., Cocanougher, B., Eichler, K., Zlatic, M., et~al. (2017).
\newblock Semiparametric spectral modeling of the drosophila connectome.
\newblock {\em arXiv preprint arXiv:1705.03297}.

\bibitem[Rand, 1971]{doi:10.1080/01621459.1971.10482356}
Rand, W.~M. (1971).
\newblock Objective criteria for the evaluation of clustering methods.
\newblock {\em Journal of the American Statistical Association},
  66(336):846--850.

\bibitem[Rohe et~al., 2011]{rohe2011}
Rohe, K., Chatterjee, S., and Yu, B. (2011).
\newblock Spectral clustering and the high-dimensional stochastic blockmodel.
\newblock {\em Ann. Statist.}, 39(4):1878--1915.

\bibitem[Schein et~al., 2016]{schein2016bayesian}
Schein, A., Zhou, M., Blei, D.~M., and Wallach, H. (2016).
\newblock Bayesian poisson tucker decomposition for learning the structure of
  international relations.
\newblock In {\em Proceedings of the 33rd International Conference on
  International Conference on Machine Learning-Volume 48}, pages 2810--2819.
  JMLR. org.

\bibitem[Sussman et~al., 2012]{sussman2012consistent}
Sussman, D.~L., Tang, M., Fishkind, D.~E., and Priebe, C.~E. (2012).
\newblock A consistent adjacency spectral embedding for stochastic blockmodel
  graphs.
\newblock {\em Journal of the American Statistical Association},
  107(499):1119--1128.

\bibitem[Sussman et~al., 2014]{6565321}
Sussman, D.~L., Tang, M., and Priebe, C.~E. (2014).
\newblock Consistent latent position estimation and vertex classification for
  random dot product graphs.
\newblock {\em IEEE Transactions on Pattern Analysis and Machine Intelligence},
  36(1):48--57.

\bibitem[Tang et~al., 2017a]{tang2017}
Tang, M., Athreya, A., Sussman, D.~L., Lyzinski, V., and Priebe, C.~E. (2017a).
\newblock A nonparametric two-sample hypothesis testing problem for random
  graphs.
\newblock {\em Bernoulli}, 23(3):1599--1630.

\bibitem[Tang and Priebe, 2018]{tang2018}
Tang, M. and Priebe, C.~E. (2018).
\newblock Limit theorems for eigenvectors of the normalized laplacian for
  random graphs.
\newblock {\em Ann. Statist.}, 46(5):2360--2415.

\bibitem[Tang et~al., 2013]{tang2013}
Tang, M., Sussman, D.~L., and Priebe, C.~E. (2013).
\newblock Universally consistent vertex classification for latent positions
  graphs.
\newblock {\em Ann. Statist.}, 41(3):1406--1430.

\bibitem[Tang et~al., 2018]{tang2016connectome}
Tang, R., Ketcha, M., Badea, A., Calabrese, E.~D., Margulies, D.~S.,
  Vogelstein, J.~T., Priebe, C.~E., and Sussman, D.~L. (2018).
\newblock Connectome smoothing via low-rank approximations.
\newblock {\em IEEE Transactions on Medical Imaging, accepted for publication}.

\bibitem[Tang et~al., 2017b]{tang2017robust}
Tang, R., Tang, M., Vogelstein, J.~T., and Priebe, C.~E. (2017b).
\newblock Robust estimation from multiple graphs under gross error
  contamination.
\newblock {\em arXiv preprint arXiv:1707.03487}.

\bibitem[Wang et~al., 2017]{wang2017joint}
Wang, S., Vogelstein, J.~T., and Priebe, C.~E. (2017).
\newblock Joint embedding of graphs.
\newblock {\em arXiv preprint arXiv:1703.03862}.

\bibitem[Young and Scheinerman, 2007]{young2007random}
Young, S.~J. and Scheinerman, E.~R. (2007).
\newblock Random dot product graph models for social networks.
\newblock In {\em International Workshop on Algorithms and Models for the
  Web-Graph}, pages 138--149. Springer.

\end{thebibliography}
\end{document}